\documentclass[preprint, 3p, 11pt, authoryear]{elsarticle}
\usepackage{graphicx}
\usepackage{amssymb}

\usepackage{lineno}
\usepackage[utf8]{inputenc}
\usepackage{algorithm}
\usepackage{algorithmicx}
\usepackage{algpseudocode}
\usepackage{amsmath}
\usepackage{amsthm}
\usepackage{dsfont}
\usepackage{enumitem}
\usepackage{etaremune}
\usepackage{natbib}
\usepackage{listings}
\usepackage{xcolor}
\usepackage{bm}
\usepackage{booktabs}
\usepackage{nicefrac}
\usepackage{graphicx}
\usepackage{multirow}
\usepackage{color, colortbl}
\usepackage{xcolor}
\definecolor{Gray}{gray}{0.9}
\usepackage{url}
\usepackage{caption}
\usepackage{subcaption}
\usepackage{enumitem}
\usepackage{microtype} 
\usepackage{multicol}
\usepackage{wrapfig}
\usepackage{framed} 
\usepackage{nomencl}
\usepackage{todonotes}

\makenomenclature
\renewcommand*\nompreamble{\begin{multicols}{2}}
\renewcommand*\nompostamble{\end{multicols}}

\newtheoremstyle{break}
 {\topsep}{\topsep}%
 {\itshape}{}%
 {\bfseries}{}%
 {\newline}{}%
\theoremstyle{break}

\newcommand{\beginsupplement}{%
 \setcounter{table}{0}
 \renewcommand{\thetable}{S\arabic{table}}%
 \setcounter{figure}{0}
 \renewcommand{\thefigure}{S\arabic{figure}}%
 }

\newcommand{\infolevel}[1]{\textbf{L}_{#1}}

\journal{European Journal of Operations Research}

\begin{document}

\newtheorem{definition}{Definition}
\newtheorem{theorem}{Theorem}
\begin{frontmatter}


\title{Policies for the Dynamic Traveling Maintainer Problem with Alerts}

\author[ieis,tue]{Paulo da Costa\fnref{foot}}
\ead{p.r.d.oliveira.da.costa@tue.nl}
\author[sor,tue]{Peter Verleijsdonk\fnref{foot}\corref{corauthor}}
\ead{p.verleijsdonk@tue.nl}
\author[ieis,tue]{Simon Voorberg\fnref{foot}}
\ead{s.voorberg@tue.nl}
\author[ieis,tue]{Alp Akcay}
\ead{a.e.akcay@tue.nl}
\author[sor,tue]{Stella Kapodistria}
\ead{s.kapodistria@tue.nl}
\author[ieis,tue]{Willem van Jaarsveld}
\ead{w.l.v.jaarsveld@tue.nl}
\author[ieis,tue]{Yingqian Zhang}
\ead{yqzhang@tue.nl}

\address[ieis]{Department of Industrial Engineering and Innovation Sciences}

\address[sor]{Department of Mathematics and Computer Science}

\address[tue]{Eindhoven University of Technology, P.O. Box 513, Eindhoven 5600 MB, the Netherlands}
\date{\today}

\fntext[foot]{These authors contributed equally.}
\cortext[corauthor]{Corresponding author}

\begin{abstract}
Downtime of industrial assets such as wind turbines and medical imaging devices comes at a sharp cost. To avoid such downtime costs, companies seek to initiate maintenance just before failure. Unfortunately, this is challenging for the following two reasons: On the one hand, because asset failures are notoriously difficult to predict, even in the presence of real-time monitoring devices which signal early degradation. On the other hand, because the available resources to serve a network of geographically dispersed assets are typically limited. In this paper, we propose a novel dynamic traveling maintainer problem with alerts model that incorporates these two challenges and we provide three solution approaches on how to dispatch the limited resources. Namely, we propose: (i) Greedy heuristic approaches that rank assets on urgency, proximity and economic risk; (ii) A novel traveling maintainer heuristic approach that optimizes short-term costs; and (iii) A deep reinforcement learning (DRL) approach that optimizes long-term costs. Each approach has different requirements concerning the available alert information. Experiments with small asset networks show that all methods can approximate the optimal policy when given access to complete condition information. For larger networks, the proposed methods yield competitive policies, with DRL consistently achieving the lowest costs.
\end{abstract}

\end{frontmatter}

\section{Introduction}\label{sec: introduction}
Industrial assets such as wind turbines, trains, medical imaging equipment, aircraft, and wafer steppers are required to experience minimal downtime. Ideally, assets are maintained just before failure so as to ensure the highest availability at minimum cost. However, the failure mechanism of such assets is \emph{a priori} typically unknown.

To keep assets functioning as intended, various maintenance policies exist: (i) Assets are maintained at scheduled times, i.e., \textit{time-based maintenance} (TBM) planning; or (ii) Assets are regularly inspected to evaluate their degradation, i.e., \textit{condition-based maintenance} (CBM). CBM is increasingly prevalent following the rise of smart monitoring devices and predictive models. Indeed, predictive models can use real-time sensor readings to predict the future failure times of assets. These predictions are typically communicated to the central operation in the form of \textit{alerts}, which serve as an early indication of potential failures \citep{TOPAN2020401}. Numerous models exist that optimize maintenance timing for individual assets. However, the ability to maintain individual assets is typically subject to the availability of maintenance resources that are shared over a network of geographically dispersed assets. In such networks, predictive models yield lists of alerts and predicted failure moments that need to be prioritized. Since alerts often carry uncertain information about the failure time due to misreadings, sensor malfunction or prediction errors, devising effective policies for these cases entails considering the uncertainty of alerts in addition to maintenance costs and travel times.

Examples of asset networks that regularly require maintenance during their long operational life include offshore wind-turbine parks. Such parks can contain over $100$ turbines and are often built in a grid-like structure, where pairwise distances between turbines are equal \citep{breton2009status}. To maintain wind turbines, maintenance engineers circulate on a vessel carrying equipment for repairs. In this case, the use of remote monitoring devices and failure predictions can represent a significant improvement over TBM policies. For example, a vibration sensor connected to the gearbox of the turbine can be used to detect signs of degradation \citep{Kenbeek2019}. Similar to wind-turbine parks, manufacturers of medical equipment (such as scanners) have to maintain their assets within a network of hospitals. Medical equipment is equipped with an array of sensors that track medical procedures, but this data can also be used to signal degradation, which in turn may trigger engineer dispatching. 

We conceptualize (what we shall refer to from now onward as) the dynamic traveling maintainer problem with alerts (DTMPA) and study the optimization of maintenance decisions in the presence of alerts, under the combination of challenges: (i) limited repair capacity and travel times; and (ii) asset degradation that is stochastic and possibly unknown. In the DTMPA, an engineer travels in a network, where each location contains a single asset. The degradation of the individual assets is modeled through: (i) A finite number of degradation states; and (ii) A set of three phases (healthy, alert and failed). The latter case corresponds to the case of censored/partial information of the degradation process. To obtain this, the degradation states are clustered in three phases and an alert is emitted when there is a transition from the \emph{healthy} phase to the \emph{alert} phase. The \textit{failed} phase corresponds to the state in which the asset has degraded to the extent that it is malfunctioning/down. 

The DTMPA is a discrete time sequential decision process.
Each unit of time, the engineer can either travel to a location, wait or start a maintenance job at its current location. Preventive maintenance restores an asset from the alert phase to the healthy phase, while corrective maintenance restores an asset from the failed phase to the healthy phase. During maintenance, the asset is down. The objective is to minimize the total expected discounted cost, which consists of the costs for (preventive and corrective) maintenance and for the machine downtime. 

When a machine transitions into the alert phase, it is valuable to use all available information to predict its residual lifetime. However, the available information may vary. Accordingly, we distinguish various levels of real-time information. These \emph{information levels} range from no further information to full information (being able to observe the degradation state and knowing the remaining lifetime distribution). Using the various information levels, the proposed DTMPA framework can be tailored to a range of use cases.

We develop and compare three solution approaches for the DTMPA. First, we introduce a class of greedy, reactive heuristics that rank actions/assets based on maintenance costs, travel times and predicted failure times. Second, we propose a heuristic that leverages the alert information to construct a deterministic approximation of the problem similar to the traveling maintainer problem (TMP) \citep{camci2014travelling}. Third, we propose an algorithm using distributional deep reinforcement learning  \citep{bellemare2017distributional} trained in a simulated environment optimizing total expected discounted maintenance costs. By using these solutions to go from alerts to decisions, we bridge the gap from predictive analytics to business prescriptive analytics for various levels of real-time degradation information.

The main contributions of the paper are detailed as follows:
\begin{itemize}
 \item We conceptualize the dynamic traveling maintainer problem with alerts (DTMPA) on dispatching an engineer in an asset network.
 \item We introduce on the DTMPA various information levels ranking the quality of the information retrieved from alerts. 
 \item We propose heuristics yielding competitive policies for their respective information level compared to the optimal policy found under complete information on the degradation process.
 \item We illustrate that deep reinforcement learning (DRL) outperforms other proposed heuristics, requiring only the minimum level of information and observations from the DTMPA.
\end{itemize}

The remainder of the paper is organized as follows. \mbox{Section \ref{sec: literature}} reviews related literature. \mbox{Sections \ref{sec: model} and \ref{sec: framework}} introduce in detail the DTMPA framework. \mbox{Section \ref{sec: solutions}} discusses the various solutions methods. In \mbox{Section \ref{sec: numerics}}, the setup of the numerical experiments is presented. \mbox{Section \ref{sec: results}} discusses the results and the corresponding managerial insights. We conclude in \mbox{Section \ref{sec: conclusion}} and provide directions for further research.
\section{Literature review}\label{sec: literature}
We review three streams of relevant literature: maintenance optimization, traveling maintainer problems and lastly deep reinforcement learning for maintenance optimization.

\paragraph{Maintenance optimization}
\emph{Maintenance optimization} models can be single-asset or multi-asset \citep{OLDEKEIZER2017405,DEJONGE2020805}. Multi-asset models extend the work of single-asset models by considering joint maintenance policies for assets that exhibit any of four types of dependencies: economic, structural, stochastic and resource dependency \citep{OLDEKEIZER2017405}. Asset degradation is typically modeled with a stochastic process defined on a discrete finite state-space, e.g., a Markov chain. 
Models with three degradation states can also be viewed as delay-time models: The second state may represent a \emph{defect} phase, which can be determined by inspection \citep{WANG2012165}. One can improve scheduled inspections with information acquired via sensors. These sensors sometimes measure asset degradation directly, but typically they measure asset health indirectly \citep{VANSTADEN2020}, for instance in the form of alerts \citep{DEJONGE201693, AKCAY2021}. In a multi-asset scenario, practical problems include complex dependencies induced by the geographical layout, giving rise to the so-called traveling maintainer problems.

\paragraph{Traveling maintainer problems}
The objective of the classical \emph{traveling maintainer problem} (TMP) is to find a path through a network of assets which minimizes the sum of the time needed to reach each asset. The TMP is obtained from the traveling salesman problem (TSP) as a mean-flow variant, and is thus NP-complete \citep{afrati1986complexity}. One can also assign a hard deadline to each asset, such as a bound on the response time, which further increases complexity. \cite{Tulabandhula2011} apply machine learning to estimate time-dependent failure probabilities at asset locations and propose two TMP objectives; learn the failure probabilities and solve the TMP sequentially or simultaneously.
 \cite{camci2014travelling} studies the TMP with the objective of minimizing the sum of functions of response times to assets. This variant, which integrates real-time CBM prognostics with the TSP by scheduling maintenance using predicted failure information, was later extended to include travel times \citep{camci2015maintenance}. The dynamic TMP (DTMP) considers service demands that arrive uniformly in some region according to a Poisson process \citep{bertsimas1989dynamic}. These service demands must be fulfilled by a single maintainer, such that the average response time is minimal. \cite{drent2020dynamic} formulate a variant of the DTMP as a Markov decision process (MDP) and propose dynamic heuristics for both the dispatching and repositioning of a repair crew using real-time condition information. \cite{HAVINGA2020107497} investigate a cyclic TMP, formulated as an MDP, combined with condition-based preventive maintenance. The proposed control-limit heuristic is compared with the optimal policy and with corrective maintenance heuristics. A challenge with MDP-based solution methods is that large asset networks induce unwieldy state spaces, rendering any exact method computationally intractable. Heuristics and function approximation methods such as deep reinforcement learning may be used to overcome these limitations. 

\paragraph{Deep reinforcement learning for maintenance optimization}
Recent DRL approaches to maintenance problems include learning opportunistic maintenance strategies on parallel machines, resulting in downtime and cost reduction compared to reactive and time-based strategies \citep{kuhnle2019reinforcement}. \cite{compare2018reinforcement} employ DRL to replacement-part management subject to stochastic failures. \cite{andriotis2020deep} and \cite{liu2020dynamic} consider inspection and maintenance policies to minimize long-term risk and costs for deteriorating assets. Similar to the DTMPA setting, this problem includes multiple challenges, such as high state cardinality, partial-observability, long-term planning, environmental uncertainties and constraints. \cite{bhattacharya2020reinforcement} consider policy rollouts and value function approximation to learn policies for a partially observable Markov decision process (POMDP) modeling a robot visiting adjacent machines. However, unlike the DTMPA case, these works do not consider maintenance optimization in an asset network. Note that DRL can potentially outperform hand-designed heuristic policies by learning to use the relevant state information from data requiring little human intervention.

\paragraph{Contribution to the literature}
In this work we combine the multiple research streams in maintenance optimization into a new framework for sequential decision-making of maintenance on asset networks under uncertainty. We consider a problem that is resource-dependent, as a single engineer maintains the entire asset network. As in previous works, we assume a finite state space per asset and allow for non-Markovian transition times between subsequent states. Moreover, we introduce three observable degradation states per asset, representing a censored observation of the true, hidden state. We assume that alerts received from sensors are related to asset degradation and contain imperfect residual lifetime predictions from black-box prediction models. 

We adopt the objective to minimize the total expected discounted cost incurred due to downtimes of the assets and (unnecessary) maintenance actions. Instead of a hard deadline, we assume that the asset-dependent costs increase if the decision-maker does not perform preventive maintenance before the deadline triggered by the asset's failure mechanism. We develop greedy heuristics based on traditional rankings of alerts and a novel traveling maintainer heuristic based on optimizing the visitation order of a constructed TMP instance considering the urgency of the alerts and the near-future costs. Given the successes of DRL in maintenance optimization, we propose a DRL algorithm based on distributional RL \citep{bellemare2017distributional}, i.e., a class of methods that approximates the distribution of long-term costs. We learn policies that consider both maintenance and traveling decisions in a stochastic environment where observations are censored and alerts act as early indicators of failures.
\section{Dynamic traveling maintainer problem with alerts}\label{sec: model}

The DTMPA is a discrete-time model that considers a single decision-maker (maintenance engineer) and a set of assets (machines), denoted by $\mathcal{M} = \{1,\ldots, M\}$, each positioned at a unique location in the network. Each asset $m \in \mathcal{M} $ degrades randomly over time and requires maintenance to prevent and resolve failures. In general, the \emph{underlying} degradation state of the asset cannot be observed directly, and the decision-maker must rely on \emph{alerts} that indicate early signs of degradation. These alerts will be used to characterize the healthy, alerted and failed phases of assets (cf. Section \ref{subsec: observations}). Each period, the decision-maker selects an action $u$: Possible actions include traveling to another location, idling/continuing or initiating maintenance at the current location. The goal is to minimize the total expected discounted cost (cf. \mbox{Section \ref{subsec: objective}}). Preventive maintenance (PM) can be carried out before a failure occurs and it restores the state of the asset to as-good-as-new. Corrective maintenance (CM) can be carried out after a failure has occurred and it also restores the state to as-good-as-new. Typically, a higher cost is paid for CM compared to PM. Maintenance takes time during which the machine is down: Failed/down machines incur downtime costs. Traveling in the network takes time proportional to the distance between machine locations. Let $\theta_{ij}\in\mathbb{N}$ denote the travel time between assets $i,j \in \mathcal{M}$. We further assume that each asset $m \in \mathcal{M}$ degrades independently according to the \textit{underlying} degradation model: In each period $t\in \mathbb{N}_0$, the asset's condition is denoted by state $x_m(t)\in \mathcal{N}_m$. Here, $\mathcal{N}_m =\{1, \ldots, |\mathcal{N}_m|\}$, where state $x^\textbf{\textrm{h}}_m=1$ denotes an as-good-as-new asset and state $x^\textbf{\textrm{f}}_m=|\mathcal{N}_m|$ denotes a failed asset. Without intervention by the decision-maker, the next state is a ``worse'' state, i.e., $x_m(t+1) \geq x_m(t)$. To transition from state $x_m\in \mathcal{N}_m	\setminus\{x^\textbf{\textrm{f}}_m\}$ to state $x_m+1$ takes a random (positive, integer) amount of time, say $T^{x_m}_m$. When asset $m$ reaches the failed state $x^\textbf{\textrm{f}}_m$, it resides in this state until a corrective maintenance action is completed.

\subsection{Observations}\label{subsec: observations}

Without intervention by the decision-maker, the random evolution of the degradation process evolves as follows: Initially, the state is set to $x^\textbf{\textrm{h}}_m$ (as-good-as-new). After a random amount of time (required to pass through states $x^\textbf{\textrm{h}}_m$ to $x^\textbf{\textrm{a}}_m$), say $T^\textbf{\textrm{a}}_m$, an alert is issued.
After another independent random amount of time (required to pass through states $x^\textbf{\textrm{a}}_m$ to $x^\textbf{\textrm{f}}_m$), say $T^\textbf{\textrm{f}}_m$ referred to as the \emph{residual lifetime}, the asset fails.

\begin{wrapfigure}{r}{0.5\textwidth}
\centering 
\includegraphics[width=1.0\linewidth]{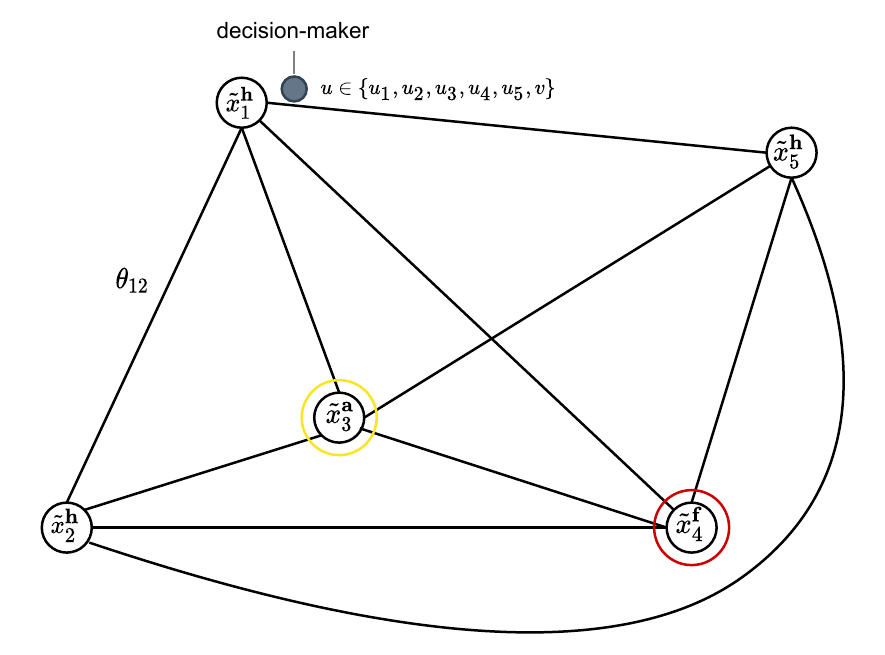}
 \caption{Visualization of the DTMPA model for an asset network of $M=5$ machines, plus the maintenance engineer. Based on observations, machines may be labeled as healthy, alerted (yellow) or failed (red). At discrete moments in time, an action $u$ is selected; possible actions include idling/ continue ($u_1$), traveling to another location ($u_2,u_3,u_4,u_5$) and initiating maintenance at the current location ($v$). }
 \label{fig:network}
\end{wrapfigure}

The decision-maker typically cannot directly observe the evolution of the degradation state $x_m\in \mathcal{N}_m$ and can only observe the three phases: (healthy, alert and failed) captured in the set of observable phases, say $\mathcal{X}_m = \{\tilde{x}^\textbf{\textrm{h}}_m, \tilde{x}^\textbf{\textrm{a}}_m, \tilde{x}^\textbf{\textrm{f}}_m\}$. The degradation states are mapped to these three phases as follows: If the degradation is in any of the states $\{x^\textbf{\textrm{h}}_m, \ldots, x^\textbf{\textrm{a}}_m-1\}$, then the phase is healthy $(\tilde{x}^\textbf{\textrm{h}}_m)$. If the degradation is in any of the states $\{ x_m^\textbf{\textrm{a}},\ldots, x^\textbf{\textrm{f}}_m-1\}$, then the phase is alert $(\tilde{x}^\textbf{\textrm{a}}_m)$. Lastly, the failed state corresponds to the failed phase $(\tilde{x}^\textbf{\textrm{f}}_m)$. 
 Based on the observable transitions alone, the degradation model can be viewed as a delay time model. See Figure \ref{fig:network} for a visualization of an asset network.

\subsection{Information levels}\label{subsec: loi}
The decision-maker receives information through alerts, which enables a categorization of machines as healthy, alerted, or failed (cf. Section~\ref{subsec: observations}). When using this information for planning purposes, some approaches may need more information regarding the underlying degradation process than others. For example, information regarding the distribution of the residual lifetime $T^\textbf{\textrm{f}}_m$ may enable more advanced planning approaches to better assess the risk of delaying maintenance on an alerted machine. We formalize this via four \textit{levels of information}:
\begin{enumerate}[label=($\textbf{L}_\arabic*$)]
\setcounter{enumi}{-1}
 \item The decision-maker acts on phase observations (cf. Section~\ref{subsec: observations}) and does not have any distributional information about the time until the alert $T^\textbf{\textrm{a}}_m$ or the residual lifetime $T^\textbf{\textrm{f}}_m$. \label{L0}

 \item The decision-maker acts on phase observations and has access to partial information (e.g., the expectation and the standard deviation) regarding the distributions $T^\textbf{\textrm{a}}_m$ and $T^\textbf{\textrm{f}}_m$.\label{L1}
 
 \item The decision-maker acts on phase observations and has full information about the underlying model (the distribution of $T^\textbf{\textrm{a}}_m$ and $T^\textbf{\textrm{f}}_{m}$ are known). \label{L2}
 
 \item The decision-maker has full state information about the underlying model and observes all state transitions (at any point in time the exact state $x_m\in \mathcal{N}_m$ is known). \label{L3}
 \end{enumerate}
Given the information level, the objective is to produce a policy that minimizes the total expected discounted cost of a given asset network. The next section formalizes the DTMPA as an MDP.

\section{Modeling framework}\label{sec: framework}
This section contains a formal description of the sequential decision-making DTMPA model framework. We introduce the underlying network states, actions, transitions, costs, observed information and the optimization objective. 

\subsection{Underlying/hidden network states}\label{subsubsec: hiddenstates}
The \textit{underlying/hidden network state} $h\in \mathcal{H}$ represents the state of all machines in the network as well as the status of the engineer. Each network state is a vector $h = (x_1, \ldots, x_M, \ell, \iota, \delta, \hat{t}_1, \ldots, \hat{t}_M )\in \mathcal{H}$. Here, $x_m\in\mathcal{N}_m$, $m\in\mathcal{M}$, denotes the degradation state of the $m$-th asset; $\ell\in\mathcal{M}$ denotes the location of the decision-maker; $\iota \in I = \{0,1\}$ indicates whether the engineer is currently performing maintenance (as opposed to traveling or waiting); $\delta\in\Delta \subset \mathbb{N}$ denotes the remaining time units until the decision-maker becomes available again ($\delta = 0$ indicates that the decision-maker is idle); $\hat{t}_m$, $m\in\mathcal{M}$, denotes the \textit{elapsed time}, i.e., the number of periods spent in the current state $x_m$. The inclusion of the elapsed times in the state space description $(h)$ ensures that the Markovian property is satisfied. Hence, $\mathcal{H}= \mathcal{N}_1 \times \ldots \times \mathcal{N}_M \times \mathcal{M} \times I \times \Delta \times \mathbb{N}^M$.

\subsection{Actions}\label{subsubsec: actions}
At every time instance, given $h= (x_1, \ldots, x_M, \ell, \iota, \delta, \hat{t}_1, \ldots, \hat{t}_M )\in \mathcal{H}$, the decision-maker can either: (i) Choose to start traveling immediately to another location, say $u_m$, $m\in\mathcal{M}\setminus\{\ell\}$; (ii) Choose to start a maintenance action at its current location, say $v$; or (iii) Continue with the current activity, including to remain idle, say $u_\ell$, with $\ell$ denoting the current location. Note that when $\delta > 0$, the engineer is busy either performing maintenance or traveling and must therefore carry out action $u_\ell$, while if $\delta=0$, then the actions can be chosen from the set $ \{u_m\}_{m\in \mathcal{M}\setminus\{\ell\}} \cup \{v\}$. The \textit{state-dependent action set} thus becomes
$$\mathcal{U}(h) = \begin{cases}
\{u_{m}\}_{m=1}^M \cup \{v\} & \textrm{if } \delta = 0, \,\\
\{u_{\ell}\} & \textrm{if } \delta > 0.
\end{cases}$$
 
\subsection{Transitions}\label{subsubsec: transitions}
The one-step state transitions $h \rightarrow h'$ are first determined by the chosen action, say $u\in\mathcal{U}(h)$, and then by the random evolution of the degradation processes. To this, $h \rightarrow h'$ is decomposed into $h \overset{u}{\rightarrow} h^u$ and to $h^u \overset{t \rightarrow t+1}\longrightarrow h'$.

For $h= (x_1, \ldots, x_M, \ell, \iota, \delta, \hat{t}_1, \ldots, \hat{t}_M ) \overset{u}{\rightarrow} h^u $,
\begin{equation*}
 h^u = \begin{cases} 
 (x_1, \ldots,\quad\quad\quad\quad\quad\quad\quad\quad x _M, \ell,\;\, \iota,\; \delta,\,\quad \hat{t}^u_1, \ldots,\quad\quad\quad\quad\quad\quad\quad \hat{t}^u_M ) & \textrm{if } u = u_m \land m = \ell, \\
 (x_1, \ldots,\quad\quad\quad\quad\quad\quad\quad\quad x_M, m, 0, \theta_{\ell m}, \hat{t}^u_1, \ldots,\quad\quad\quad\quad\quad\quad\quad \hat{t}^u_M ) & \textrm{if } u = u_m \land m \neq \ell, \\
 (x_1, \ldots, x_{\ell-1}, x^\textbf{\textrm{f}}_\ell, x_{\ell+1}, \ldots, x_M, \ell,\;\, 1,\, t^\textrm{PM}_\ell,\, \hat{t}^u_1, \ldots, \hat{t}^u_{\ell-1}, 0, \hat{t}^u_{\ell+1}, \ldots, \hat{t}^u_M ) & \textrm{if } u = v \land x_\ell \neq x^\textbf{\textrm{f}}_\ell, \\
 (x_1, \ldots,\quad\quad\quad\quad\quad\quad\quad\quad x_M, \ell,\;\, 1,\, t^\textrm{CM}_\ell,\, \hat{t}^u_1, \ldots,\quad\quad\quad\quad\quad\quad\quad \hat{t}^u_M ) & \textrm{if } u = v \land x_\ell = x^\textbf{\textrm{f}}_\ell.
 \end{cases}
\end{equation*}

In the first case, the decision-maker stays at the current location $m=\ell$ and carries on with the ongoing action, including to remain idle, $(u_\ell)$. In the second case, the decision-maker initiates travel. The remaining unavailability $\delta$ changes to the travel time $\theta_{\ell m}$ and the location $\ell$ changes to the destination $m\in\mathcal{M}\setminus\{\ell\}$. The third and fourth case represent initiating maintenance at the current location, where the maintenance type (PM or CM) depends on the status $x_{\ell}$ of the machine at the location of the engineer. The duration of maintenance is either $t^{\textrm{PM}}_{m} \in \mathbb{N}^+$ or $t^{\textrm{CM}}_{m} \in \mathbb{N}^+ $, and $\delta$ changes accordingly. Initiating maintenance on machine $m$ advances the degradation state $x_m$ to $x^\textbf{\textrm{f}}_m$, indicating that the machine is unavailable during maintenance, which is in line with the modeling assumptions.

For $h^u = (x^u_1, \ldots, x^u_M, \ell^u, \iota^u, \delta^u, \hat{t}^u_1, \ldots, \hat{t}^u_M )
\overset{t \rightarrow t+1}\longrightarrow h' = (x'_1, \ldots, x'_M, \ell', \iota', \delta', \hat{t}'_1, \ldots, \hat{t}'_M )$, the determination of $h'$ is performed in two steps. In step 1, we check for every machine if there is further degradation or completion of maintenance. In step 2, we update the remaining state variables of the model. The state changes as follows:
\begin{enumerate}
\item The machines evolve according to the following triptych of cases:
If $(\ell^u, \iota^u, \delta^u) = (m, 1, 1)$, then $(x_m', \hat{t}_{m}')=(x^\textbf{\textrm{h}}_m, 0)$. This first case represents completion of a maintenance job, resulting in a transition to the as-good-as-new state. 
Else, with probability $\mathbb{P}(\hat{t}^u_m):= \mathbb{P}(T_m^{x^u_m} = \hat{t}^u_m+1| T_m^{x^u_m} > \hat{t}^u_m)$, $(x_m', \hat{t}_{m}')=(x^u_m+1,0) $. 
This second case represents the transition to a subsequent degradation state and $\mathbb{P}(\hat{t}^u_m) $ denotes the corresponding probability. 
Otherwise, with probability $1-\mathbb{P}(\hat{t}^u_m)$, $(x_m', \hat{t}_{m}')=(x^u_m, \hat{t}^u_m+1) $. This third case captures that the machine degradation state does not change, which implies that the elapsed time since the last transition increases by one.
\item $ (\ell^u, \textrm{max}(\delta^u - 1,0), \iota^u) \overset{t \rightarrow t+1}\longrightarrow (\ell', \delta', \iota')$; i.e., the remaining unavailability $\delta'$ of the engineer, due to an ongoing action, is decreased by one, if possible.
\end{enumerate}

\subsection{Cost structure}\label{subsubsec: costs}
Initiating maintenance on machine $m\in\mathcal{M}$ incurs costs $c^{\textrm{PM}}_{m} \in\mathbb{R}^+$ or $c^{\textrm{CM}}_{m} \in\mathbb{R}^+$, for PM or CM, respectively. We assume that $c^{\textrm{CM}}_{m} \geq c^{\textrm{PM}}_{m}$. An asset $m$ is said to be down when it is in the failed state or it is under repair, viz. in state $x^\textbf{\textrm{f}}_m$. The cost of downtime is $c^{\textrm{DT}}_{m}\in\mathbb{R}^+$ per time unit, independently of the cause of downtime. Summing up, when taking action $u\in\mathcal{U}(h)$ in state $h$, the incurred costs are:
\begin{equation}
\label{eq: network_costs}
 C(h,u) = c^{\textrm{CM}}_{\ell}\mathds{1}_{\{u = v, x_\ell = x^\textbf{\textrm{f}}_\ell\}} + c^{\textrm{PM}}_{\ell}\mathds{1}_{\{u = v, x_\ell < x^\textbf{\textrm{f}}_\ell\}}+ \sum_{m\in\mathcal{M}} c^{\textrm{DT}}_{m}\mathds{1}_{\{ x_m = x^\textbf{\textrm{f}}_m \}}.
\end{equation}

\subsection{Observed network states and censoring}\label{subsubsec: ObservedStates}
For information levels $\textbf{L}_0$ to $\textbf{L}_3$, the underlying state $h = (x_1, \ldots, x_M, \ell, \iota, \delta, \hat{t}_1, \ldots, \hat{t}_M )$ is not fully observable. Instead, one observes $o = (\tilde{x}_1, \ldots, \tilde{x}_M, \ell, \iota, \delta, \tilde{t}_1, \ldots, \tilde{t}_M) \in \Omega = \mathcal{X}_1 \times \ldots \times \mathcal{X}_M \times \mathcal{M} \times I \times \Delta \times \mathbb{N}^M$, where, for $m\in \mathcal{M}$, $\tilde{x}_m\in\mathcal{X}_m$, cf. Section \ref{subsec: observations}, and where $\tilde{t}_m$ corresponds to the time since machine $m$ transitioned to observable phase $\tilde{x}_m$.

\subsection{Objective }\label{subsec: objective}
We are interested in a policy, say $\pi^\textbf{L}$, with $\textbf{L} \in \{\textbf{L}_0, \textbf{L}_1, \textbf{L}_2, \textbf{L}_3\}$ denoting the information level, which minimizes the total expected discounted cost. More precisely, a policy is defined as a sequence of decision rules, i.e., $\pi^\textbf{L} = (\pi_1^\textbf{L}, \pi_2^\textbf{L}, \ldots,\pi_t^\textbf{L}, \ldots )$, where the decision rule $\pi_t^\textbf{L}$ assigns the probability of the action to be taken at time $t$, given the history (observed states from time 0 until time $t$ and the respective actions). Let $\gamma \in [0,1)$ be the discount factor and $J(\pi^\textbf{L})$ be the total expected discounted cost. Then, the objective is to find the optimal policy $\pi^{\textbf{L}}_*$ satisfying
\begin{equation}
\label{eq: objective}
 \pi^{\textbf{L}}_* = \textrm{arg }\underset{\pi^\textbf{L}}{\textrm{min }} J(\pi^\textbf{L}) = \textrm{arg }\underset{\pi^\textbf{L}}{\textrm{min }} \lim_{T\to\infty} \mathbb{E}_{\pi^\textbf{L}} \Bigg[ \sum_{t=0}^T \gamma^t C\left(O(t),U(t)\right) \Bigg| O(0)=o\Bigg],
\end{equation}
where $(O(t),U(t))$ denotes the tuple of the observed state and the respective action given the policy $\pi^\textbf{L} $ at time $t$, $t\geq0$, and $C(\cdot)$ denotes the associated cost (maintenance and downtime).
\section{Solution approaches}\label{sec: solutions}
In this section, we discuss three solution approaches for the introduced class of problems.

\subsection{Ranking heuristics} \label{sec: greedy}
In this section, we construct a simple class of $\infolevel{1}$ heuristics: The \emph{reactive} heuristic, which maintains only failed machines, and the \emph{greedy} heuristic, which maintains both alerted and failed machines. To select which machine to maintain next, both heuristics require the introduction of a \emph{ranking} between assets. The ranking is made on the premise of failures and on the premise of alerts, accounting also for travel times and costs.

{\em Failure time \textsc{(F)}.} Given the alert information, the machines are ranked on their estimated failure time. Assets in the failed state are prioritized over any of these potential future failures.
Let $t\in\mathbb{N}_0$ be the current time and $t^\textbf{\textrm{a}}_{m}$ be the time at which the last alert for the machine was generated, and define the failure time estimate $\tilde{T}^\textbf{\textrm{f}}_{m} = \textrm{max}\left\{t, t^\textbf{\textrm{a}}_{m} + \mathds{E}\left[T^\textbf{\textrm{f}}_{m}\right] \right\}$. 

{\em Travel time \textsc{(T)}.} Assets are ranked based on proximity to the engineer. 
For distant assets, the travel times will be long, meaning valuable time will be lost while no maintenance is conducted. The decision-maker opts for maintaining assets over traveling. 

{\em Cost savings \textsc{(C)}.} To account for assets with different cost structures (economic risk), we adopt a ranking that is constructed as follows: For assets in the alert state, suppose asset $m$, we compute the immediate benefit of PM over CM, specifically $c_m = (c^{\textrm{CM}}_{m} - c^{\textrm{PM}}_{m}) + (t^{\textrm{CM}}_{m} - t^{\textrm{PM}}_{m}) c^{\textrm{DT}}_{m}$. For assets in the failed state, suppose $m'$, we compute the total downtime cost when deciding to directly repair the asset, given by $ c_{m'} = (\theta_{\ell, m'} + t^{\textrm{CM}}_{m'}) c^{\textrm{DT}}_{m'}$. 
The aim of this ranking is to prioritize the machines which carry the highest economic risk.

For both heuristics, the assets to be maintained are ranked based on the three rankings in a chosen order, e.g., [\textsc{F},\textsc{T},\textsc{C}]. The reactive heuristic (in contrast to the greedy heuristic) is obtained by ranking only the failed assets. Heuristics like the ones presented above are a reasonable choice for companies faced with problems similar to the DTMPA; we restrict attention to two variants here, a range of ranking heuristics could be further identified.

\subsection{Traveling maintainer heuristic}

The \textit{traveling maintainer heuristic} (TMH) policy is inspired by the traveling maintainer problem (TMP) that arises for a group of assets which require maintenance at deterministic maintenance deadlines. Namely, the TMH policy globally consists of the following two steps: 

\emph{First order approximation.} In this step, we approximate the random times that an alerted asset might fail by deterministic maintenance deadlines. Using these approximations, the problem reduces to a deterministic TMP.

\emph{Optimization step.} Solve the constructed TMP instance by optimizing the \textit{schedule}, i.e., the time at which maintenance is initiated for all assets, for each possible order in which the assets can be visited. A schedule is optimal when it achieves the minimum discounted cost in the deterministic optimization problem; ties are broken by selecting uniformly at random from the schedules with the longest duration.

When applying the THM policy at time $t\in\mathbb{N}_0$, we first construct a \emph{first order approximation}. We compute deterministic maintenance deadlines, say $t_m$, for the various failed and alerted assets. For failed assets, we set $t_m$ equal to the corresponding failure time $t^{\textbf{\textrm{f}}}_m\le t$. For alerted assets, it may be advantageous to postpone maintenance, as maintenance incurs costs and induces costly downtime. Accordingly, for alerted assets, we compute $t_m$ as the cost-optimal maintenance time $\tau$ associated with a time-based maintenance (TBM) policy. Note that effective DTMPA downtime costs depend on the \emph{response time} (the time needed before the engineer starts maintenance at the asset location), which is $0$ when $|\mathcal{M}| = 1$, but which is difficult to characterize when $|\mathcal{M}|>1$. Assuming a response time of $0$ results in a $\tau$-TBM policy with total expected discounted cost given by:
\begin{equation}\label{eq: disctmbJ}
 J(\tau) = \frac{ \bar{c}_m^{\textrm{CM}}\mathds{E}[\gamma^{T^\textbf{\textrm{a}}_m + T^\textbf{\textrm{f}}_m}\mathds{1}_{\{ T^\textbf{\textrm{f}}_m \leq \tau \}} \mid X_m(0) = x^\textbf{\textrm{h}}_m] + \bar{c}_m^{\textrm{PM}}\mathds{E}[\gamma^{T^\textbf{\textrm{a}}_m + \tau}\mathds{1}_{\{ T^\textbf{\textrm{f}}_m > \tau \}} \mid X_m(0) = x^\textbf{\textrm{h}}_m] }{ 1 -\gamma^{t^\textrm{CM}_m} \mathds{E}[\gamma^{T^\textbf{\textrm{a}}_m + T^\textbf{\textrm{f}}_m}\mathds{1}_{\{ T^\textbf{\textrm{f}}_m \leq \tau \}} \mid X_m(0) = x^\textbf{\textrm{h}}_m] - \gamma^{t^\textrm{PM}_m}\mathds{E}[\gamma^{T^\textbf{\textrm{a}}_m + \tau}\mathds{1}_{\{ T^\textbf{\textrm{f}}_m > \tau \}} \mid X_m(0) = x^\textbf{\textrm{h}}_m] }.
 \end{equation}
 
A proof of this claim can be found in \ref{app:theorems}. The optimal $\tau$-TBM policy\footnote{The \emph{average cost criterion} can be computed by evaluating the limit of $(1-\gamma)J(\tau)$ as $\gamma\rightarrow 1$.} $\tau^*$ can be found by numerically computing $\tau^* = \textrm{arg min } J(\tau)$. Now, for an alerted machine, say $m$, with corresponding alert time $t_a$, we set the deadline as:
\begin{equation}\label{eq: tmh-4}
t_m = \textrm{max}\left\{t, t_a + \tau^* + 1 \right\}.
\end{equation}

It is now assumed that asset $m$ fails at time $t_m$, thus maintenance is preferably initiated at time $t_m-1$. Accordingly, the \emph{optimization step} plans a cost-efficient route and schedule to maintain the failed and alerted assets. 

Let $\mathcal{M}_t$ denote the subset of all failed and alerted machines at time $t$ and let $\Sigma_{\mathcal{M}_t}$ denote the permutations of $\mathcal{M}_t$. For each path $\sigma \in \Sigma_{\mathcal{M}_t}$, we construct an initial schedule (cf. Algorithm \ref{alg: tmh-1} in \ref{app:tmhalgos}) that is tight in the sense that travel actions and maintenance actions are executed without intermediate delay. To improve this schedule, we may introduce intermediate delay/slack. To simplify slack assignment, we identify conditions such that any voluntary violation of maintenance deadlines is sub-optimal in the constructed TMP. Define the combined costs $\Bar{c}_m^{\textrm{CM}} = c_m^{\textrm{CM}}+ c_m^{\textrm{DT}}\frac{\gamma^{t_m^{\textrm{CM}}} - 1 }{\gamma-1}$ and $\Bar{c}_m^{\textrm{PM}} = c_m^{\textrm{PM}}+c_m^{\textrm{DT}}\frac{\gamma^{t_m^{\textrm{PM}}} - 1 }{\gamma-1}$, where $\sum_{i=1}^t \gamma^{i-1} = \frac{\gamma^t-1}{\gamma-1}$. Voluntary violations of maintenance deadlines are suboptimal if the following four conditions hold for all $m$: (i) $\gamma\bar{c}^\textrm{CM}_m > \bar{c}^\textrm{PM}_m$ (it is loss-making to perform CM instead of PM), (ii) $c^\textrm{DT}_m + \gamma \bar{c}^\textrm{CM}_m > \bar{c}^\textrm{CM}_m$ (it is loss-making to postpone a CM action), (iii) $\gamma\bar{c}^\textrm{CM}_m - \bar{c}^\textrm{PM}_m > \sum_{m^\prime \in \mathcal{M}\backslash\{m\}} (1-\gamma)\bar{c}^\textrm{PM}_{m^\prime}$ (it is loss-making to perform CM instead of PM to delay an arbitrary number of other PM actions) and (iv) $c^\textrm{DT}_m + (\gamma-1) \bar{c}^\textrm{CM}_m > \sum_{m^\prime \in \mathcal{M}\backslash\{m\}} (1-\gamma)\bar{c}^\textrm{PM}_{m^\prime}$ (it is loss-making to delay a CM action to delay an arbitrary number of PM actions).

To find an optimal schedule for a given route, we thus first delay the last repair as much as possible, which will allow us to delay the previous repair as well, and so on. See Algorithm \ref{alg: tmh-2} in \ref{app:tmhalgos} for details. Given a path $\sigma$ and the optimized schedule $\textbf{t}^\sigma_\textrm{opt}$, the tuple $(\sigma, \textbf{t}^\sigma_\textrm{opt})$ induces a policy in the DTMPA, which induces a trajectory of observations $o_t, \ldots, o_{\textbf{t}^\sigma(2|\mathcal{M}_t|+1)}$ for which we can compute the discounted cost. Select the tuple $(\sigma, \textbf{t}^\sigma_\textrm{opt})$ that minimizes the discounted costs.

The heuristic advantages are numerous. Firstly, the heuristic considers jointly the network layout, cost structures and the information captured in the alerts. Secondly, one only needs to solve a TMP when a new alert or failure arrives, otherwise the previous solution can continue to be used for planning actions. A disadvantage concerns the algorithm complexity as $|\Sigma_{\mathcal{M}_t}| = |\mathcal{M}_t|!$, which in the worst case implies a run time of $\mathcal{O}(2^{M})$. 

The TMH policy picks the path where it meets the most costly maintenance deadlines, preferably scheduling maintenance at time $t_m-1$ but possibly expediting this decision to ensure that as many maintenance deadlines as possible are met. The policy can adequately assess the immediate benefit of PM over CM, but it may potentially be less suitable when response times are substantial and downtime costs are the dominant factor. 

\subsection{Deep reinforcement learning}\label{sec: drl}

We adopt Eq.(\ref{eq: objective}) as a DRL objective under $\textbf{L}_0$. In the corresponding POMDP formulation (cf. Section~\ref{subsubsec: ObservedStates}), observable network states summarize the relevant \emph{history} in the form of elapsed times since the last observable transition. DRL is applicable when the proposed environment can be simulated for a number of episodes $E \in \mathbb{N}$. Note that this is possible when learning online by observing past state transitions. We could also reuse transitions from previous policies or implement a simulation environment that mimics real-world data before deploying the learned agents.

We seek to learn a policy $\pi^{\textbf{L}_0}$, mapping observable network states to action probabilities. 
Learning is possible by sampling a number of trajectories following a policy, and evaluating its performance with respect to the main objective in Eq.(\ref{eq: objective}). 
To evaluate performance, we define the standard characterization of the state-action value function measuring the expected sum of costs starting from a given observable network state $o$, action $u$, starting at time $t$ as 
\begin{equation}
 q^{\pi^{\textbf{L}_0}} (o,u) = \mathbb{E}_{\pi^{\textbf{L}_0}} \Bigg[ \sum_{k=0}^\infty \gamma^k C\left(O(t+k),U(t+k)\right) | O(t) = o, U(t) = u \Bigg].
\end{equation}

To avoid notation clutter, we refer to the class of DRL policies interchangeably as $\pi^{\textbf{L}_0}$ and $\pi$, costs at time $t$ as $c_t := C\left(O(t),U(t)\right)$ and observable network states as $o_t := O(t)$. 
We assume no access to the environment transition probabilities (model), in which case a policy $\pi$ may be learned as a mapping from states to actions or extracted from value function approximations. We propose a value (distribution) approximation method, referred to as $n$-step quantile regression double Q-Learning ($n$QR-DDQN) detailed in the forthcoming sections.

\subsubsection{$n$-step quantile regression double Q-Learning } \label{subsec:deep-q}

Similar to \citep{hessel2018rainbow,badia2020agent57}, we extend deep Q-Learning (DQN)\citep{mnih2015human} by combining several modifications that improve performance, training times and maintain comparable sampling complexity \citep{hessel2018rainbow,dabney2018distributional,van2016deep}, detailed below. In DQN, off-policy RL is combined with neural network (NN) approximation to estimate values of state-action pairs. State information is passed to a NN $q_w : \Omega \times \mathcal{U} \rightarrow \mathds{R} $, where $w$ are trainable parameters. Similar to the original DQN, we store a set of transitions $(o_t, u_t, c_t, o_{t+1})$ in a \textit{replay memory} $\mathcal{D}$ (hash table). Transitions are collected following an $\epsilon$-greedy exploration scheme, i.e., picking a random action with probability $\epsilon$, otherwise selecting the action with the lowest estimated $q$ values. For training, we sample a mini-batch of $N_B \in \mathds{N}$ transitions from $\mathcal{D}$ uniformly at random to decorrelate past transitions.
Based on the mean Bellman optimality operator, the parameters of the NN are trained to minimize the one-step Temporal Difference (TD[0]) error \citep[p.~131]{sutton2018reinforcement}, i.e., the error between the current estimate of the value of a state and its one-step update, acting greedily with respect to the approximated $q$ function via a mean squared error (MSE) loss function. DQN introduces a \emph{target network} with parameters $\bar{w}$, i.e., a periodic copy of the online $q_w$ that is not directly optimized and used as a target for the future $q$-value estimates. This copy stabilizes training and represents a fixed target that does not change at each update of $w$ \citep{mnih2015human}. This copy is updated every $P \in \mathds{N}$ episodes. 

\paragraph{Double deep Q-Learning}
DQN can be affected by an underestimation bias (overestimation in maximization) \citep{van2016deep}. In a noisy environment, this can lead to over-optimistic actions and prevent learning. Thus, we employ double DQN (DDQN) \citep{van2016deep} to decouple the action selection from its evaluation.
In this way, the action selection when updating $w$ remains dependent on the online $q_w$ (the current network being optimized), effectively reducing the chance of underestimating action values on the target network $q_{\bar{w}}$. DDQN effectively changes the loss of DQN to
\begin{equation}
 \mathcal{L}(w) = (c_{t} + \gamma q_{\bar{w}} (o_{t+1},\arg \min_{u^\prime} q_w (o_t,u^\prime) ) - q_w (o_t,u_t))^2. 
\end{equation}

\paragraph{Multi-step reinforcement learning} 
One-step temporal difference updates of DQN are biased when $q_w$ estimates are far from true $q^\pi$ values. To obtain better estimates, TD[0] updates can be generalized by bootstrapping over longer horizons. This reduces bias but it can lead to high variance due to the environment's stochasticity \citep{jaakkola1994convergence}. Nevertheless, using a larger bootstrapping horizon can empirically lead to higher performance and faster learning. We include longer horizons considering the truncated $n$-step future discounted costs 
\begin{equation}
\label{eq: truncated-costs}
 C^{t:t+n} = c_t + \gamma c_{t+1} + \gamma^2 c_{t+2} + \ldots + \gamma^{n-1} c_{t+n-1}. 
\end{equation}
Using such costs makes the updates of DQN on-policy and results in value function updates that rely on old and inaccurate transitions in the replay memory. 
When the policy that generates the transitions in the replay memory differs substantially from the current policy, the resulting updates can cause inaccurate evaluations of the $q$ function under the current policy. To make updates off-policy again, we need to introduce importance sampling terms \citep{de2018multi}. However, these terms can also lead to higher variance. In practice, adding $n$-step terms (even without importance sampling) can aid learning, while high variance can be controlled by small values of $n$ \citep{hernandez2019understanding}. 
Thus, to balance bias (one-step updates) and variance (multi-step updates), we include an additive $n$-step term weighted by $\alpha \in \mathds{R}$ in the original DQN training objective. The parameter $\alpha$ is introduced to control how much extra variance we allow when updating the estimated $q$ values.

\paragraph{Distributional reinforcement learning} 
We consider the failure of crucial assets, in which risk-adjusted costs need to be considered when selecting actions, thus learning a distribution of future costs allows us to improve policy performance by taking into account the stochasticity of the problem explicitly. 
Distributional RL \citep{bellemare2017distributional} aims to learn not a single point estimate of values but a distribution of returns, i.e., the distribution of the random variable $z^\pi(o,u) = \sum_{k=0}^\infty \gamma^k c_{t+k} $ for given $o$ and $u$, where $q^\pi(o,u) := \mathbb{E}_\pi[z^\pi(o,u)]$
\citep{bellemare2017distributional}. We employ quantile regression DQN (QR-DQN) \citep{dabney2018distributional}, in which the distribution of returns is modeled via a quantile regression on $N$ data points with fixed uniform weights of the CDF of $z^\pi$ as $\tau_i = \frac{i}{N}, \, i=1, \ldots, N$.

Thus, we estimate the quantiles by learning a model $\psi_w: \Omega \times \mathcal{U} \rightarrow \mathds{R}^N$, mapping each state-action pair to a probability distribution supported on $\{\psi^i_{w}(o,u)\}^N_{i=1}$, i.e., $ z_\psi(o,u) := \frac{1}{N} \sum^N_{i=1} \delta_{\psi^i_{w}(o, u)}$, 
where $\delta_z$ represents a Dirac at $z \in \mathds{R}$. Moreover, each $\psi^i_w$ represents an estimation of the quantile value corresponding to the CDF quantile weight $\hat{\tau}_i=\frac{\tau_{i-1}+\tau_{i}}{2}$, $\tau_0 = 0$, i.e., the data points that minimize the 1-Wasserstein metric to the true $z^\pi$ \citep{dabney2018distributional}. Note that $z_\psi(o,u)$ can be used to compute the usual $q_w$ estimates as $q_w(o, u) = \frac{1}{N}\sum^N_{i=1} \psi^i_{w}(o, u)$. To achieve unbiased gradients in QR-DQN, we replace the usual MSE loss of DQN with an asymmetric variant of the Huber loss \citep{huber1964} defined as $ \rho^\kappa_{\hat{\tau}_i}(\nu) = |\hat{\tau}_i - \mathds{1}_{\{\nu<0\}}| \mathcal{L}_{\kappa}(\nu)$, where
\begin{equation}
 \mathcal{L}_\kappa (\nu) = \begin{cases}
 \frac{1}{2}{\nu^2} & \text{for } |\nu| \le \kappa, \\
 \kappa (|\nu| - \frac{1}{2}\kappa), & \text{otherwise,}
\end{cases}
\end{equation}
 $\nu$ corresponds to pairwise TD errors and $\kappa \in \mathds{R}$. Similar to the mean Bellman optimality operator, the distributional Bellman optimality operator \citep{bellemare2017distributional}, i.e.,
 $ z^\pi(o,u) =C_u(o)+ \gamma z^\pi(o^\prime, \arg \min_{u^\prime} \mathds{E}_\pi[z^\pi(o^\prime, u^\prime)]),$ where $o^\prime$ is the next observable network state, is used to approximate quantiles based on the $n$-step loss function (with DDQN correction) as:
 
\begin{equation}
 \mathcal{L}^{t:t+n}_{QR}(w) = \frac{1}{N} \sum^N_{i=1} \sum^N_{i^\prime=1}\rho^\kappa_{\hat{\tau}_i} (y^{i^\prime}_{t:t+n} - \psi^i_{w}(o_t,u_t)),
\end{equation}

where $y^{i^\prime}_{t}= C^{t:t+n} + \gamma^n \psi^{i^\prime}_{ \bar{w}}(o_{t+n}, \arg \min_{u^\prime} \sum^N_{i=1} \psi^i_{w}(o_{t+n}, u^\prime))$. Finally, we add the $n$-step QR loss term to a single-step term to arrive at the objective:
\begin{equation}
 \label{eq:qrddqn-nstep-loss}
 w^* = \arg \min_{w} \hat{\mathcal{L}}_{\text{$n$QR-DDQN}}(w) = \arg \min_{w} \left[\mathcal{L}_{QR}(w) + \alpha \mathcal{L}_{QR}^{t:t+n}(w)\right],
\end{equation}
note that $\mathcal{L}_{QR}(w)$ is the same as $\mathcal{L}_{QR}^{t:t+1}(w)$ and corresponds to TD[0] (one-step) updates. This is the training objective for $n$QR-DDQN with the complete algorithm depicted in \ref{app:rlalgo}.
\begin{figure}[ht]
\centering \includegraphics[width=1\linewidth]{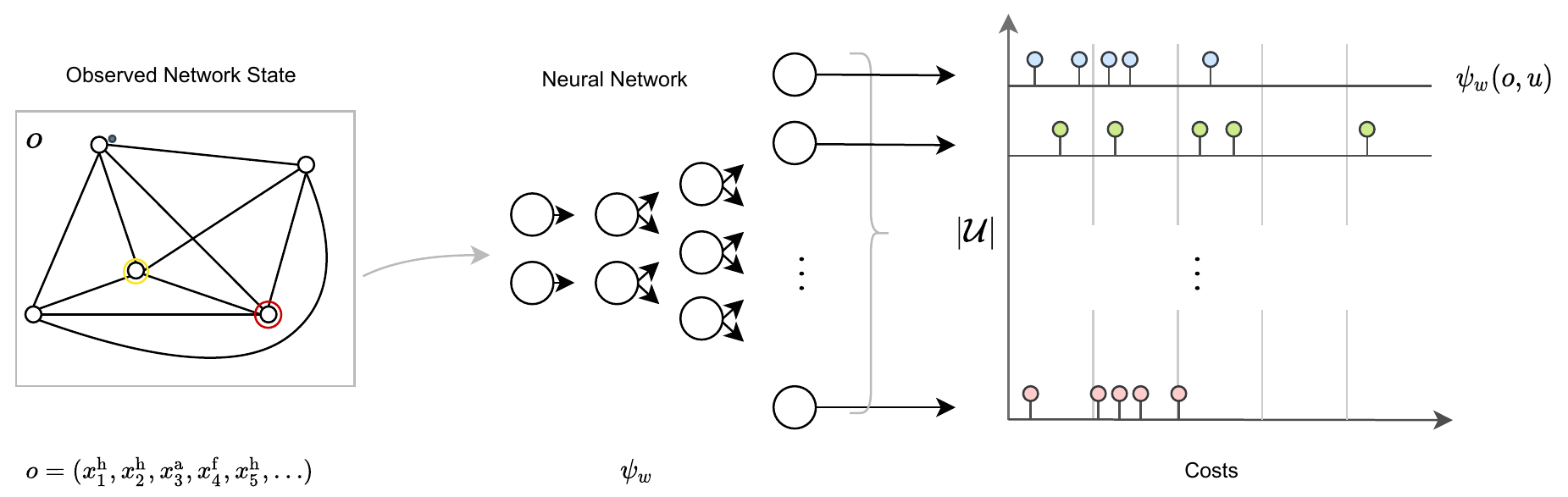}
 \caption{A neural network $\psi_w: \mathds{R}^d \rightarrow \mathds{R}^{N\times |\mathcal{U}|}$, mapping observable state vectors to estimates of the quantile locations of the cost distributions.}
 \label{fig:NN}
\end{figure}
\paragraph{Neural network architecture} \label{subsec: deep-nn}
We employ a NN $\psi_w$, depicted in Figure \ref{fig:NN}, that maps raw observable network state vectors to pairs of actions and 
$N$ quantiles of the distribution of $z^\pi$, i.e., $\psi_w: \mathds{R}^d \rightarrow \mathds{R}^{N\times |\mathcal{U}|}$. First, an observable network state is flattened into an input vector $\bm{\zeta} \in \mathds{R}^{d}$, where $d=3M+2$.
The vector $\bm{\zeta}$ is then passed through a series of layers $l \in \{1, \ldots, L\}$ of the form $ \bm{\zeta}^l = \sigma^l(\mathbf{W}^l\bm{\zeta}^{l-1}+\bm{b}^l)$, 
where $\mathbf{W}^l \in \mathds{R}^{d^{l} \times d^{l-1} }$, $\bm{b}^l \in \mathds{R}^{d^{l}}$, $d^l \in \mathds{N}$, $d^0 = d$, $\bm{\zeta}^{0} = \bm{\zeta} $ and $\sigma^l$ is a (non-)linear activation function in every hidden layer, i.e., $l \in \{1, \ldots, L-1\}$. In the output layer $L$, $d^L = |\mathcal{U}| \times N $, where $|\mathcal{U}|$ is the size of the action space and $\sigma^L$ is the identity function. The parameters $w = \{\mathbf{W}^l, \bm{b}^l\}^L_{l=1}$ are updated via the gradient descent method, Adaptive Moment Estimation (Adam) \citep{kingma2015adam}, computing adaptive learning rates for each parameter, with learning rate $\lambda > 0$ aimed at minimizing the loss in Eq.(\ref{eq:qrddqn-nstep-loss}). Note that $\lambda$ controls the magnitude of the gradient updates performed in Adam.
\section{Experimental settings}\label{sec: numerics}

To benchmark the algorithms discussed in Section \ref{sec: solutions}, we construct several asset networks varying the number of machines $M \in \{1, 2, 4, 6\}$ in a network, while reducing complexity by assuming that $t^\textrm{PM}_m = t^\textrm{CM}_m = \theta_{mm^\prime} \equiv 1, \forall m \neq m^\prime \in \mathcal{M}$ and that machine degradation times are geometrically distributed (see Section~\ref{sec:machinedegrad}). Under information level $\infolevel{3}$, the DTMPA then reduces to a rather manageable MDP: We can solve problems of up to four machines to optimality using policy iteration. With this exact benchmark in hand, we will be better positioned to assess the performance of the approaches developed in Section~\ref{sec: solutions}. Note that those approaches operate on information levels $\infolevel{0}-\infolevel{2}$, for which the DTMPA becomes a POMPD with infinite state space which resists exact analysis. A detailed setup of the experiments follows in the remainder of this section. Experimental results are discussed in the following section. 

\subsection{Machine degradation}\label{sec:machinedegrad}
Machine degradation models in the experimental setup share the following characteristics. The alert state is the first non-healthy state, that is, $x^\textbf{\textrm{a}}_m \equiv x^\textbf{\textrm{h}}_m+1=2$. Transition times $T^{x_m}_m$ between states $x_m \in \mathcal{N}_m\setminus\{x^\textbf{\textrm{f}}\}$ and $x_m+1$ are geometrically distributed, i.e., $T^{x_m}_m \sim \textrm{Geo}(p^{x_m}_m)$, where $p^{x_m}_m \in (0,1)$ for all $m$ and $x_m$. We assume $p^{x_m}_m = p^\textbf{\textrm{a}}_m$ for all $x_m \in \{ x^\textbf{\textrm{a}}_m, \ldots, x^\textbf{\textrm{f}}_m-1\}$, thus $T^\textbf{\textrm{f}}_m = \sum_{k=2}^{x^\textbf{\textrm{f}}_m} \textrm{Geo}(p^\textbf{\textrm{a}}_m) \overset{\textrm{d}}{=} x^\textbf{\textrm{f}}_m-2 + \textrm{NegBin}(x^\textbf{\textrm{f}}_m-2, p^\textbf{\textrm{a}}_m)$. Geometrically distributed transition times enable us to capture degradation as a simple matrix $Q$ (cf. \cite{derman1963optimal}), where $Q_{x_m,x_m}$ denotes the per-period probability to remain in degradation state $x_m$, while $Q_{x_m,x_m+1}$ denotes the probability to transition to state $x_m+1$. Adopting this notation, we consider four types of machine degradation, denoted as follows:
\[\small{
\textrm{Q1} = \begin{bmatrix}
0.8 & 0.2 & 0\\
0 & 0.7 & 0.3\\
0 & 0 & 1\\
\end{bmatrix}\quad
\textrm{Q2} = \begin{bmatrix}
0.8 & 0.2 & 0 & 0 & 0\\
0 & 0.7 & 0.3 & 0 & 0\\
0 & 0 & 0.7 & 0.3 & 0 \\
0 & 0 & 0 & 0.7 & 0.3 \\
0 & 0 & 0 & 0 & 1 
\end{bmatrix}
\label{mat:q2}}
\]
\\
\[\small{
\textrm{Q3} = \begin{bmatrix}
0.8 & 0.2 & 0 & 0 & 0\\
0 & 0.3 & 0.7 & 0 & 0\\
0 & 0 & 0.3 & 0.7 & 0 \\
0 & 0 & 0 & 0.3 & 0.7 \\
0 & 0 & 0 & 0 & 1 
\end{bmatrix}\quad
\textrm{Q4} = \begin{bmatrix}
0.8 & 0.2 & 0 & 0 & 0 & 0 & 0\\
0 & 0.7 & 0.3 & 0 & 0 & 0 & 0 \\
0 & 0 & 0.7 & 0.3 & 0 & 0 & 0 \\
0 & 0 & 0 & 0.7 & 0.3 & 0 & 0 \\
0 & 0 & 0 & 0 & 0.7 & 0.3 & 0 \\
0 & 0 & 0 & 0 & 0 & 0.7 & 0.3 \\
0 & 0 & 0 & 0 & 0 & 0 & 1 \\
\end{bmatrix}
\label{mat:q4}}
\]
\\
Degradation model \textrm{Q1} has only three states, and hence $\mathcal{N} = \mathcal{X}$ rendering the model fully-observable. For machines that degrade following \textrm{Q1}, approaches developed in Section~\ref{sec: solutions} could in theory achieve the performance of the optimal policy computed using policy iteration (PI) under information level $\textbf{L}_{3}$. That is, the optimal solution to the MDP is in the class of $\textbf{L}_{0}$ policies. This is no longer true if machines that degrade following matrices \textrm{Q2}-\textrm{Q4} are present. Degradation models \textrm{Q2} and \textrm{Q3} both transition to an alert state with probability $0.2$, but \textrm{Q2} degrades slower afterwards, compared to \textrm{Q3}. Machine \textrm{Q4} has more underlying states, while per-period degradation is comparable to \textrm{Q2}.
\subsection{Cost structure}
We define three cost structures, namely \textrm{C1}, \textrm{C2} and \textrm{C3}, which are presented in Table \ref{tab:cost-ratio}. Each cost structure represents a distinct, realistic cost relationship between preventive and corrective costs that induces unique optimal policies favouring more or less frequent maintenance actions.
\begin{table}[]
\centering
\caption{Cost structures considered in the experiments.}
\label{table:tsplib}
\resizebox{0.55\textwidth}{!}{%
\begin{tabular}{c|ccc|c}
\toprule
Cost Structure & $c^{\textrm{PM}}$ & $c^{\textrm{CM}}$ & $c^{\textrm{DT}}$ & $\nicefrac{c^{\textrm{CM}}+c^{\textrm{DT}} }{c^{\textrm{PM}}+c^{\textrm{DT}} }$ \\ \midrule
\textrm{C1} & 0 & 9 & 1 & 10 \\
\textrm{C2} & 1 & 2 & 10 & 1 \\
\textrm{C3} & 1 & 4 & 1 & 2.5 \\
 \bottomrule
\end{tabular}
}
\label{tab:cost-ratio}
\end{table}
For example, when $ \nicefrac{\Bar{c}^{\textrm{CM}} }{\Bar{c}^{\textrm{PM}} }$ is large, i.e., when CM costs greatly surpass PM costs, we expect that preventive maintenance policies outperform reactive policies. 

\subsection{Asset networks}
We introduce $16$ \textit{asset networks}, each having a different combination of network size, cost structure and machine degradation matrices. 
We create various asset configurations for which a shorthand notation is adopted: First the number of assets is given, followed by the degradation matrices which are assumed to be spread evenly across the machines. For example, M4-Q2Q3 has four machines: Two with degradation matrix Q2 and two with matrix Q3.

We briefly discuss specifics regarding the instances. The simplest setup is M1-Q1; for this instance $\mathcal{X}_m = \mathcal{N}_m$ such that $\infolevel{3}$ has no informational advantage compared to $\infolevel{0}$. Network M1-Q4 considers a machine (Q4) where $\mathcal{X}_m \subsetneq \mathcal{N}_m$, resulting in informational advantage of $\infolevel{3}$ over the other information levels. To scale up the model complexity, we consider M2-Q2Q3 which considers two machines with different degradation rate, and M4-Q2Q3 which is obtained by doubling the M2-Q2Q3 network. The most challenging problems are obtained for M6-Q2Q3Q4: Apart from containing the most assets, note that the time between alert and failure is relatively long for Q4 machines, inducing a time-based maintenance problem where it is essential to adequately consider the risk of postponing maintenance, especially when maintenance resources are limited as is the case in the DTMPA. 

These five asset configurations are considered with cost structures C1, C2 and C3. We obtain a final configuration M6-Q2Q3Q4-C by considering a custom cost structure on M6-Q2Q3Q4: Q2 machines are assigned cost structure C2, Q3 machines the cost structure C3 and Q4 machines have cost structure C1. In total, this gives rise to $5\times3+1 = 16$ DTMPA instances. 

\subsection{Reinforcement learning training parameters}
For all experiments, we simulate for a maximum of $T=500$ time steps\footnote{Assuming a maximum error level $\xi(1-\gamma)/c \approx 0.006$ between the infinite horizon costs and truncated horizon costs at time $T$, i.e., $T > \log(\xi (1-\gamma)c^{-1}) / \log(\gamma ) - 1$ where $c = 10$.} and $E=2000$ episodes. Each new simulation episode is started with different alert/failure arrivals following identical probability distributions. We employ a NN comprised of a linear embedding and two layers with non-linear activations. All models are trained on a Ryzen 3950X CPU and a RTX 2080Ti GPU. The remaining parameters settings are presented in \ref{app:drlparam}.

\section{Experimental results}\label{sec: results}
We generate 512 test episodes for each experimental setting. In the experiments, if the same deterministic policy is evaluated twice on an episode, the performance will be identical. 
\begin{table*}[!ht]
\centering
\caption{Average costs (95\%CI) over 512 episodes and 500 steps. 
\textbf{Bold:} Lowest cost of the proposed approaches.
}
\label{tab:avgresults}
\resizebox{1\textwidth}{!}{%
\begin{tabular}{c|l|ccc}

\toprule
\multicolumn{1}{c|}{\multirow{2}{*}{\textbf{Solution (Info. Level)}}} &
 \multicolumn{1}{c|}{\multirow{2}{*}{\textbf{Asset Network}}} &
 \multicolumn{3}{c}{\textbf{Cost Structure}} \\ \cmidrule(l){3-5} 
\multicolumn{1}{c|}{} &
 \multicolumn{1}{c|}{} &
 \multicolumn{1}{c}{\textbf{C1 {[}95\% CI{]}}} &
 \multicolumn{1}{c}{\textbf{C2 {[}95\% CI{]}}} &
 \textbf{C3 {[}95\% CI{]} } \\ \midrule

 \multirow{4}{*}{PI ($\textbf{L}_3$)} 

&
 M1-Q1 &
 16.36 &
 123.91 &
 32.72 \\
 &
 M1-Q4 &
 4.730 &
 47.582 &
 9.461 \\
 &
 M2-Q2Q3 &
 21.230 &
 190.275 &
 39.550 \\
 &
 M4-Q2Q3 &
 79.976 &
 432.440 &
 96.166 
 \\\midrule
\multirow{6}{*}{$n$QR-DDQN ($\textbf{L}_0$)} &
 M1-Q1 &
 \textbf{16.365} {[}16.171, 16.56{]} &
 \textbf{124.96} {[}123.541, 126.378{]} &
 \textbf{32.804} {[}32.443, 33.165{]} \\
 &
 M1-Q4 &
 \textbf{8.806} {[}8.608, 9.004{]} &
 \textbf{47.408} {[}46.894, 47.922{]} &
 \textbf{14.513} {[}14.343, 14.683{]} \\
 &
 M2-Q2Q3 &
 \textbf{25.139} {[}24.935, 25.343{]} &
 \textbf{202.311} {[}200.675, 203.946{]} &
 \textbf{46.995} {[}46.609, 47.381{]} \\
 &
 M4-Q2Q3 &
 \textbf{92.654} {[}90.736, 94.571{]} &
 \textbf{470.625} {[}467.640, 473.611{]} &
 \textbf{106.525} {[}105.717, 107.334{]} \\
 &
 M6-Q2Q3Q4 &
 \textbf{176.642} {[}175.234, 178.051{]} &
 \textbf{711.188} {[}705.789, 716.586{]} &
 \textbf{159.527} {[}158.294, 160.760{]} \\
 &
 M6-Q2Q3Q4-C &
 \multicolumn{3}{c}{\textbf{347.500} {[}344.986, 350.014{]}} \\ \midrule
\multirow{6}{*}{H - Greedy {[}\textsc{F},\textsc{T},\textsc{C}{]} ($\textbf{L}_1$)} &
 M1-Q1 &
 \textbf{16.365} {[}16.171, 16.56{]} &
 182.233 {[}180.126, 184.339{]} &
 \textbf{32.804} {[}32.443, 33.165{]} \\
 &
 M1-Q4 &
 16.61 {[}16.417, 16.804{]} &
 179.75 {[}177.624, 181.876{]} &
 32.818 {[}32.474, 33.163{]} \\
 &
 M2-Q2Q3 &
 30.9 {[}30.495, 31.305{]} &
 306.366 {[}304.26, 308.472{]} &
 56.692 {[}56.279, 57.105{]} \\
 &
 M4-Q2Q3 &
 112.304 {[}110.395, 114.212{]} &
 526.248 {[}523.62, 528.877{]} &
 112.306 {[}111.444, 113.168{]} \\
 &
 M6-Q2Q3Q4 &
 231.498 {[}228.491, 234.505{]} &
 741.568 {[}735.639, 747.497{]} &
 168.064 {[}166.677, 169.451{]} \\
 &
 M6-Q2Q3Q4-C &
 \multicolumn{3}{c}{379.799 {[}375.934, 383.665{]}} \\ \midrule
\multirow{6}{*}{H - Reactive {[}\textsc{F},\textsc{T},\textsc{C}{]} ($\textbf{L}_1$)} &
 M1-Q1 &
 103.361 {[}102.217, 104.506{]} &
 \textbf{124.96} {[}123.541, 126.378{]} &
 51.736 {[}51.195, 52.278{]} \\
 &
 M1-Q4 &
 40.018 {[}39.595, 40.442{]} &
 \textbf{47.408} {[}46.894, 47.922{]} &
 20.127 {[}19.912, 20.342{]} \\
 &
 M2-Q2Q3 &
 154.074 {[}153.033, 155.114{]} &
 283.619 {[}281.469, 285.768{]} &
 82.419 {[}81.845, 82.993{]} \\
 &
 M4-Q2Q3 &
 306.278 {[}304.876, 307.68{]} &
 718.158 {[}713.699, 722.617{]} &
 173.682 {[}172.799, 174.565{]} \\
 &
 M6-Q2Q3Q4 &
 396.714 {[}395.106, 398.321{]} &
 1053.663 {[}1046.581, 1060.745{]} &
 231.742 {[}230.677, 232.806{]} \\
 &
 M6-Q2Q3Q4-C &
 \multicolumn{3}{c}{473.647 {[}470.884, 476.41{]}} \\ \midrule
\multirow{6}{*}{TMH ($\textbf{L}_2$)} &
 M1-Q1 &
 	 
 \textbf{ 16.365} {[}16.171, 16.56{]} &
 \textbf{124.96} {[}123.541, 126.378{]} 	 &
 \textbf{32.804} {[}32.443, 33.166{]} \\
 &
 M1-Q4 &	
 \textbf{8.806} {[}8.608, 9.004{]} &
 \textbf{47.408} {[}46.894, 47.922{]} &
 \textbf{14.513} {[}14.343, 14.683{]} \\
 &
 M2-Q2Q3 &
 \textbf{ 25.221} {[}25.037, 25.405{]} 	 	 
 &
 235.746 {[}233.683, 237.809{]} &
 \textbf{46.757 } {[}46.404, 47.111{]} \\
 &
 M4-Q2Q3 &
 111.591 {[}110.188, 112.993{]}	&
634.828 {[}630.347, 639.309{]} &111.865 {[}110.988, 112.743{]} \\
 &
 M6-Q2Q3Q4 &
 214.435 {[}212.463, 216.408{]} &
 989.006 {[}982.159, 995.853{]} &
 181.077 {[}179.628, 182.527{]} \\
 &
 M6-Q2Q3Q4-C &
 \multicolumn{3}{c}{ 379.669 {[}376.796, 382.543{]}		} \\ \bottomrule
\end{tabular}%
}
\end{table*}
We measure the solutions' performance and present the estimated total expected discounted cost together with a $95\%$ confidence interval, after rolling out the policies for $T=500$ steps on each episode. We also include the results obtained finding the optimal policy under information level $\infolevel{3}$ following policy iteration \citep{puterman1990markov}. We were able to solve all instances with $|\mathcal{M}| \le 4$ using \cite{cartesius}, the Dutch national supercomputer. The averaged results are summarized in Table \ref{tab:avgresults}; a detailed investigation revealed that solutions that yield identical performance were in fact identical. We briefly discuss results for the various asset configurations:

\paragraph{M1-Q1} There is no information gap, and the proposed solutions match the performance and behavior of the policies found via PI. Moreover, the optimal policies coincide with the greedy heuristic (C1 and C3) and the reactive heuristic (C2).

\paragraph{M1-Q4} The optimal policy of the underlying MDP prescribes to either wait until the last state before failure (C1 and C3) or to delay maintenance to when the asset has failed (C2). Note that the latter policy is reactive, and indeed the proposed reactive heuristic is optimal. 
The best performing solutions are the TMH and $n$QR-DDQN. Since both operate under partial information of the failure probabilities, they obtain similar results by either computing an optimal TBM policy or learning via interaction. 

\paragraph{M2-Q2Q3} When increasing the number of machines, we expect the performance gap between the $\infolevel{3}$ policy (obtained via PI) and the other policies to grow. In this experiment, the learned agent achieves the best performance for all cost structures. The TMH policy obtains similar performance to $n$QR-DDQN for cost structures C1 and C3. Otherwise, the TMH policy shows poor performance, since it is does not account for positive response times. 

\paragraph{M4-Q2Q3 - M6-Q2Q3Q4} 
When the network size increases, the heuristics fail to cope with the increased complexity. The best performing heuristic (C1 and C3) is the TMH policy, showing that more information yields better policies, even if they are myopic. However, it consistently shows poor performance when machines with a C2 cost structure are part of the asset network. 
For M4-Q2Q3, $n$QR-DDQN achieves similar performance to the optimal policies under full information. It is noteworthy that the DRL agent learns to neglect one Q3 machine for M6-Q2Q3 with cost structure C1, potentially accounting for the capacity limit of the maintenance engineer.

\paragraph{M6-Q2Q3Q4-C} 
Both the greedy heuristic and the TMH solution yield similar performance. This is due to more machines residing in the observed alert state simultaneously leading the TMH to push maintenance to meet as many deadlines as possible. Thus, the benefit of having access to a residual lifetime metric is reduced. 
Nevertheless, the DRL agent outperforms all heuristics, implying that the agent is able to delay maintenance decisions more accurately. 

\subsection{Policy comparison}
In this section, we focus on asset network M2-Q2Q3. This particular experiment serves as a tractable setting to understand where the proposed policies for networks containing more than one asset differ.

\subsubsection{State visitation distribution}
We select M2-Q2Q3-C1 as an instance to provide deeper insights on how frequently the various network states are visited for each policy. This latter quantity will be referred to as visitation distribution. For an instance with two machines, we can represent a censored network state using the presence of alerts (\textbf{H}: Healthy, \textbf{A}: Alert, \textbf{F}: Failed). To reduce the size of the state space, we omit the decision-maker's location and the elapsed times. Under these simplifications, we can examine the visitation distribution of each policy. 
\begin{figure*}[]
 \centering
 \begin{subfigure}[t]{0.23\textwidth}
 \centering
 \includegraphics[width=\textwidth]{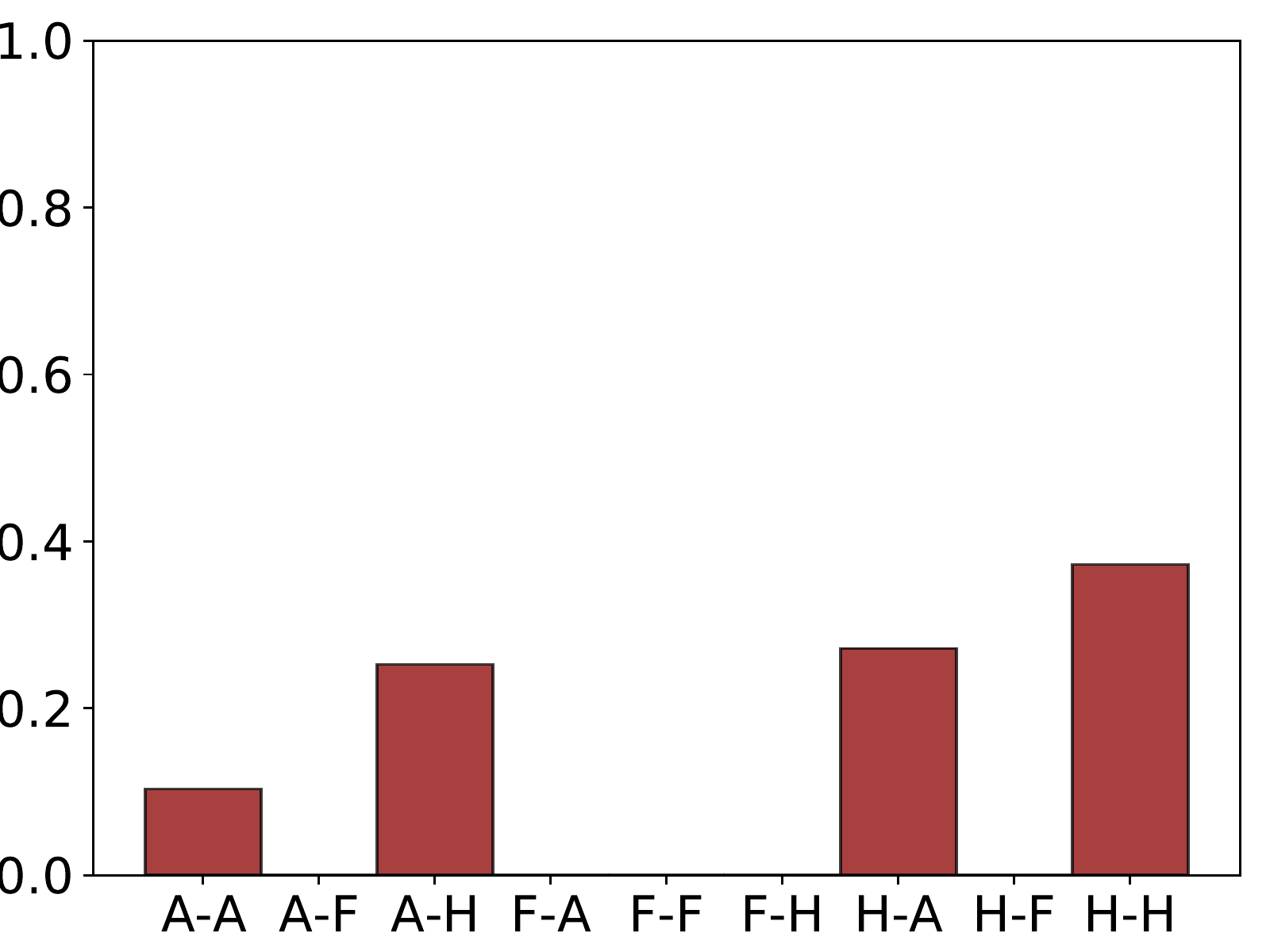}
 \caption{$n$QR-DDQN $J(\pi)$:25.139}
 \label{subfig:dqn}
 \end{subfigure}%
 \hfill
 \begin{subfigure}[t]{0.23\textwidth}
 \centering
 \includegraphics[width=\textwidth]{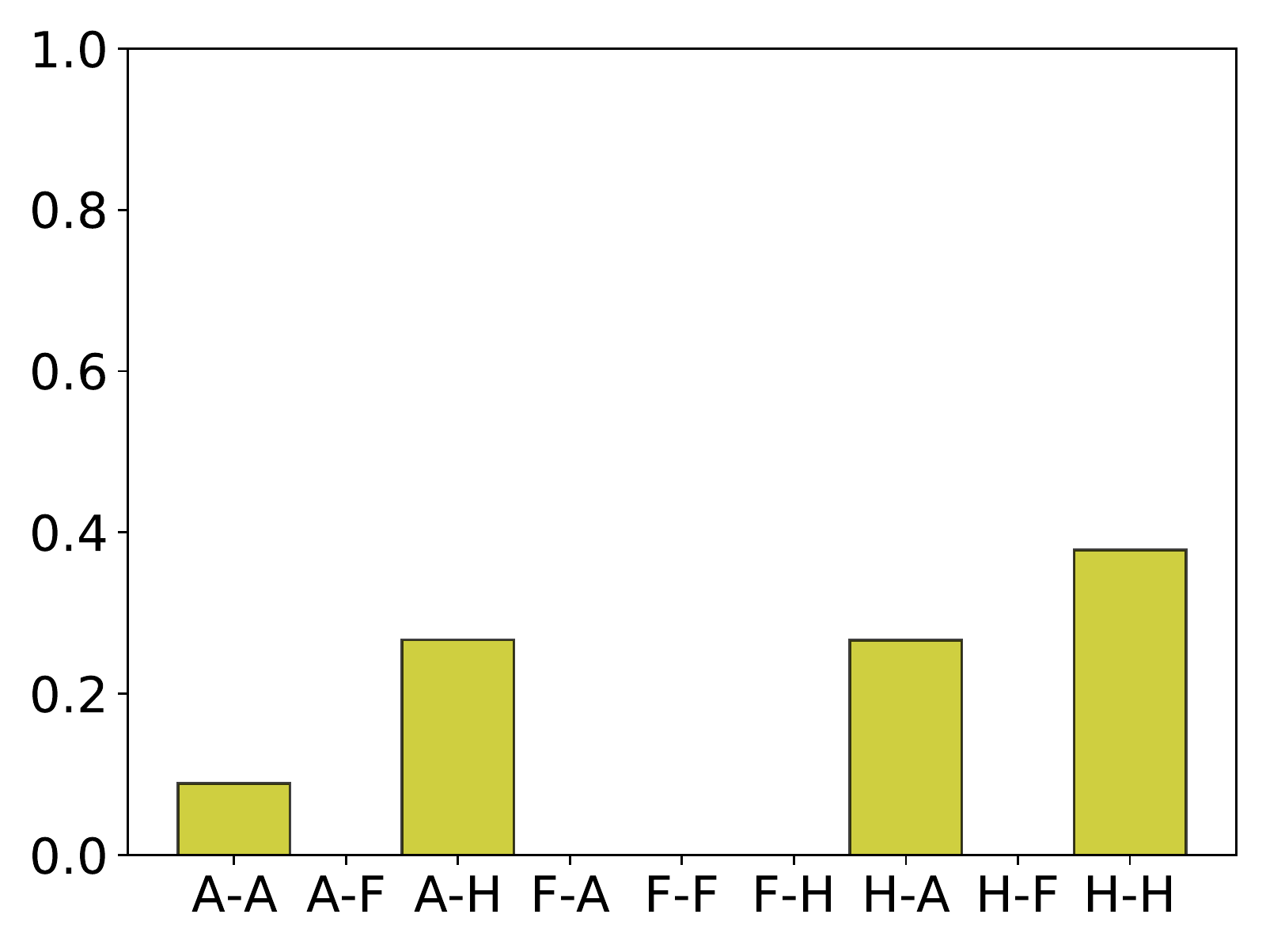}
 \caption{TMH. $J(\pi)$:25.221}
 \label{subfig:da}
 \end{subfigure}
 \hfill
 \begin{subfigure}[t]{0.23\textwidth}
 \centering
 \includegraphics[width=\textwidth]{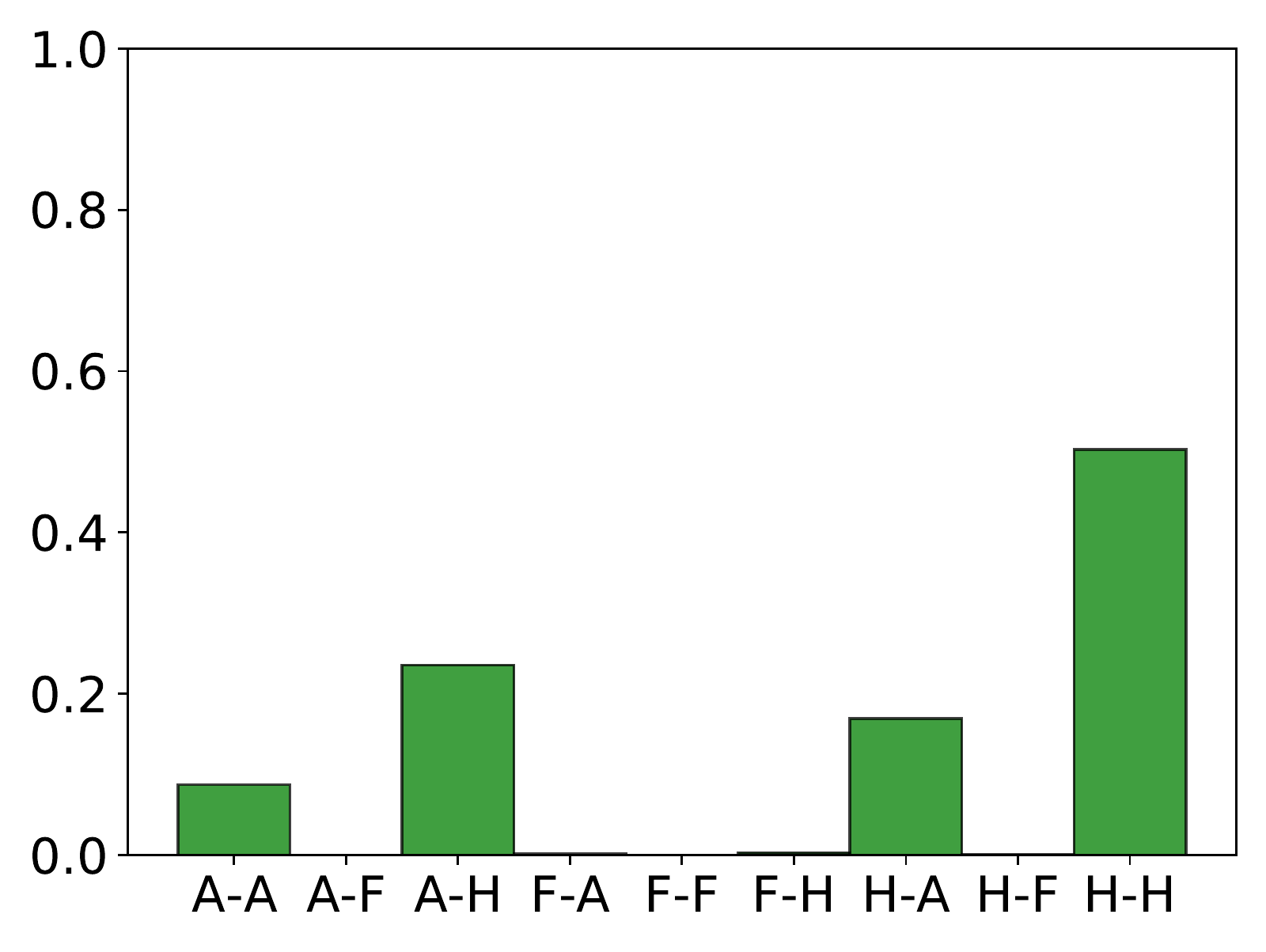}
 \caption{H-Greedy. $J(\pi)$:30.9}
 \label{subfig:greedy}
 \end{subfigure}
 \hfill
 \begin{subfigure}[t]{0.23\textwidth}
 \centering
 \includegraphics[width=\textwidth]{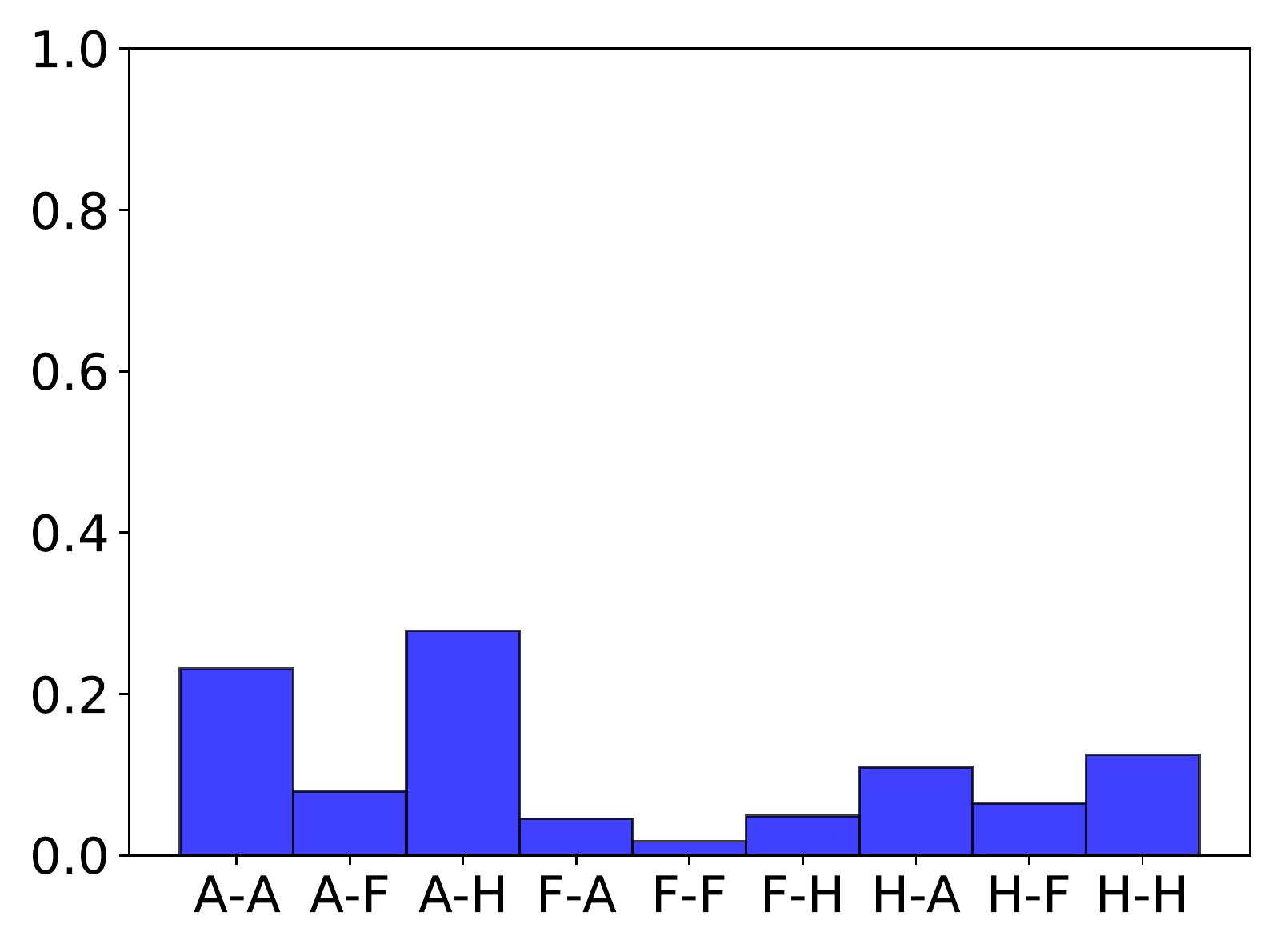}
 \caption{H-Reactive. $J(\pi)$:154.074}
 \label{subfig:reac}
 \end{subfigure}
 \caption{Visitation distribution of censored network states over all episodes of problem instance M2-Q2Q3-C1.\\(\textbf{H}: Healthy, \textbf{A}: Alert, \textbf{F}: Failed). }
 \label{fig:state-dist}
\end{figure*}
Figure \ref{fig:state-dist} shows the visitation distribution when rolling out each policy for $512$ episodes. The reactive heuristic (Figure \ref{subfig:reac}) visits more frequently the network states where one or both machines are in the failed state, which leads to high CM costs under the C1 cost structure. The greedy heuristic (Figure \ref{subfig:greedy}) mostly visits network states where both machines are in a healthy state (H), implying that unnecessary costs are incurred due to excessive PM. Since the benefit of PM over CM in C1 is large, this policy yields lower costs than the reactive heuristic. Both the TMH (Figure \ref{subfig:da}) and $n$QR-DDQN (Figure \ref{subfig:dqn}) have similar visitation distributions since both these policies attempt to delay maintenance to the ``right" time. However, the heuristic is more conservative than the learned policy, i.e., the DRL agent visits network states where the second machine is an alert state (H-A) more often. This shows that the DRL agent takes more ``risk", allowing the second machine to be in the alert state slightly longer compared to the TMH.

\subsubsection{Policy behavior}
M2-Q2Q3-C2 is used as an informative instance to study the differences between policies in a given episode, in order to relate the actions taken by each policy and the performance implications.

In Figure \ref{fig:policy-diff}, we present the time step $t=3$ for a single episode of M2-Q2Q3-C2. In this example, the machine at location 0 is of type Q2, and the machine at location 1 is of type Q3. Moreover, under cost structure C2, reactive actions in the network are relatively attractive. As such, the reactive heuristic outperforms the greedy heuristic due to fewer repairs. In the example of Figure \ref{subfig:pol-reac}, we observe that in the presence of an alert in machine location 1, the action taken by the reactive heuristic is to wait (\textsc{pass}) until the machine fails (at $t=6$) before traveling to that location. Under the greedy heuristic, (Figure \ref{subfig:pol-greedy}), the decision-maker already traveled to location 1 when the alert arrived (at $t=2$), and the action at $t=3$ is to perform preventive maintenance (\textsc{pm}) at that location, incurring the PM costs of repairing that machine. Due to the cost structure, repeatedly choosing to do PM over CM causes the greedy heuristic to accumulate higher costs.
\begin{figure*}
 \centering
 \begin{subfigure}[t]{0.24\textwidth}
 \centering
 \includegraphics[width=\textwidth]{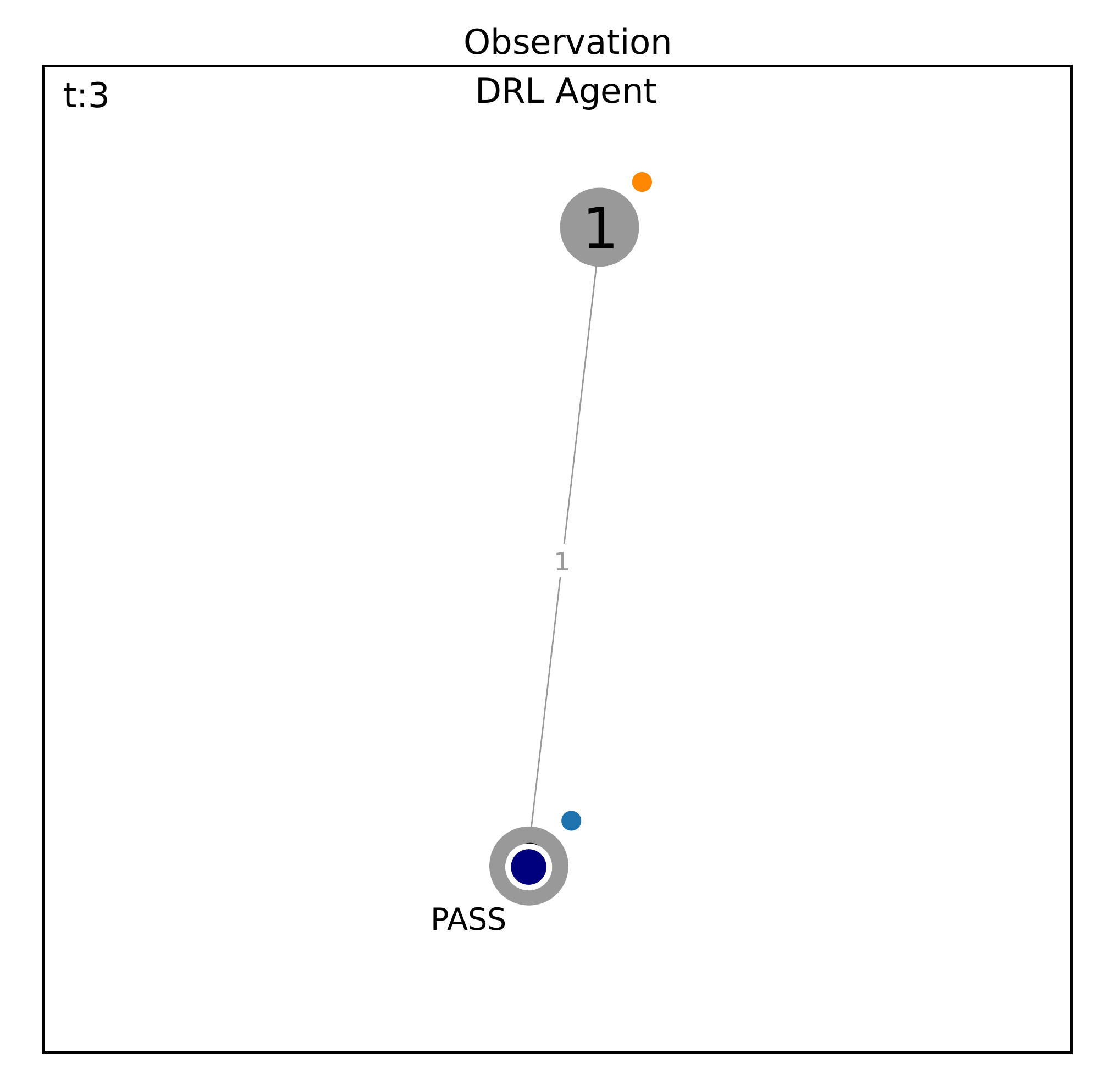}
 \caption{$n$QR-DDQN. $J(\pi)$:202.311}
 \label{subfig:pol-dqn}
 \end{subfigure}%
 \hfill
 \begin{subfigure}[t]{0.24\textwidth}
 \centering
 \includegraphics[width=\textwidth]{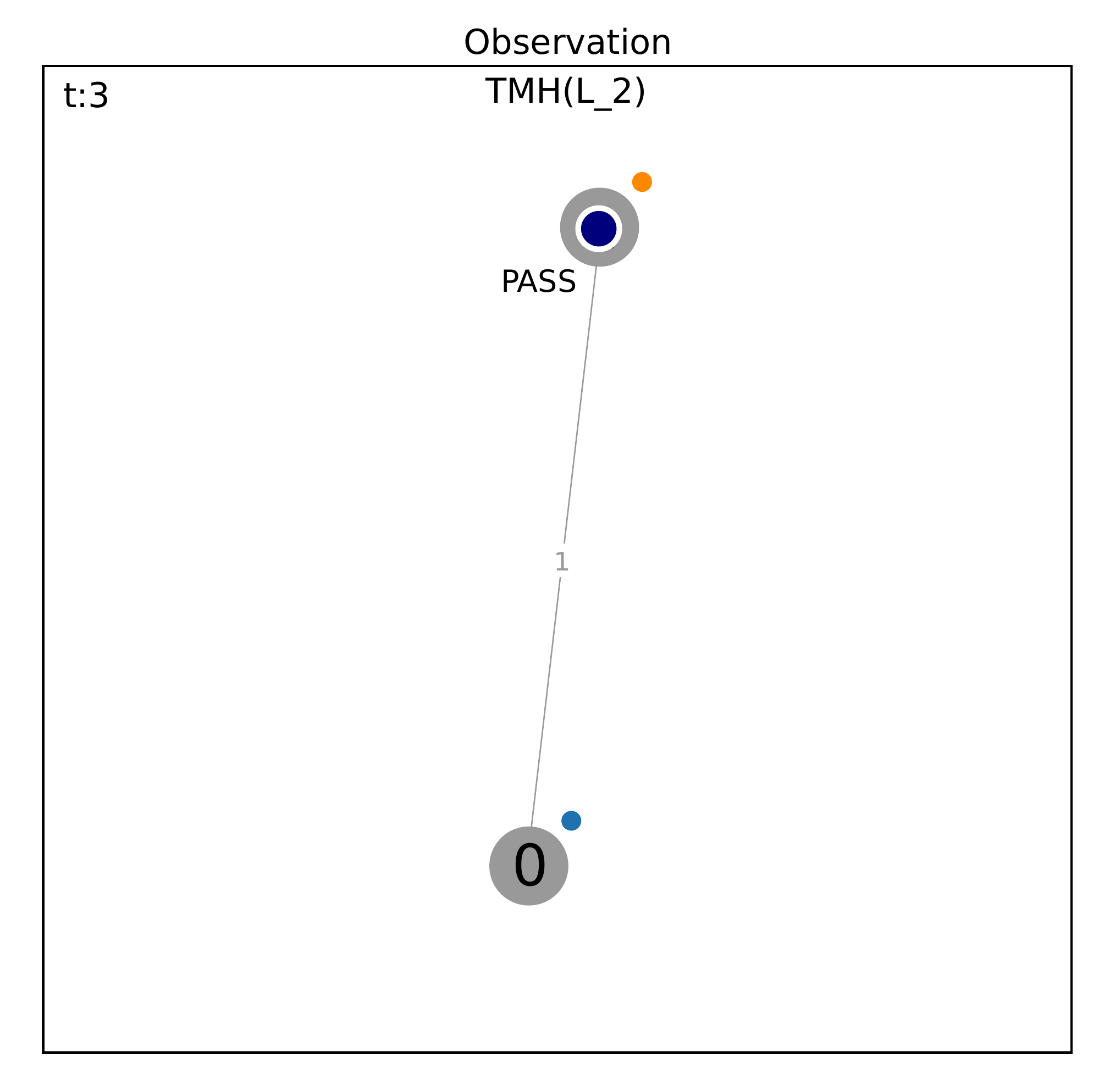}
 \caption{TMH. $J(\pi)$:235.746}
 \label{subfig:pol-da}
 \end{subfigure}
 \hfill
 \begin{subfigure}[t]{0.241\textwidth}
 \centering
 \includegraphics[width=\textwidth]{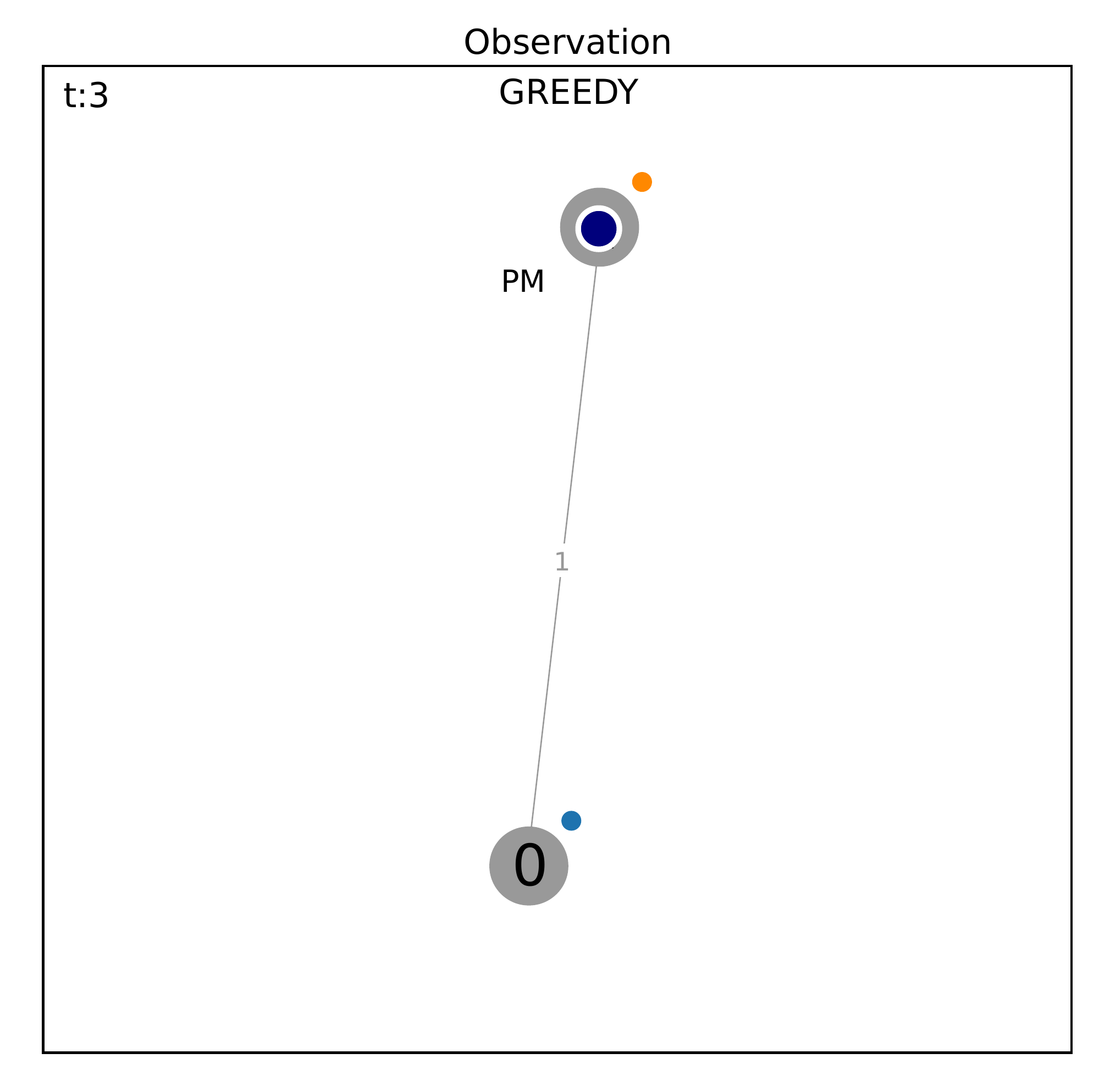}
 \caption{H-Greedy. $J(\pi)$:306.366}
 \label{subfig:pol-greedy}
 \end{subfigure}
 \hfill
 \begin{subfigure}[t]{0.24\textwidth}
 \centering
 \includegraphics[width=\textwidth]{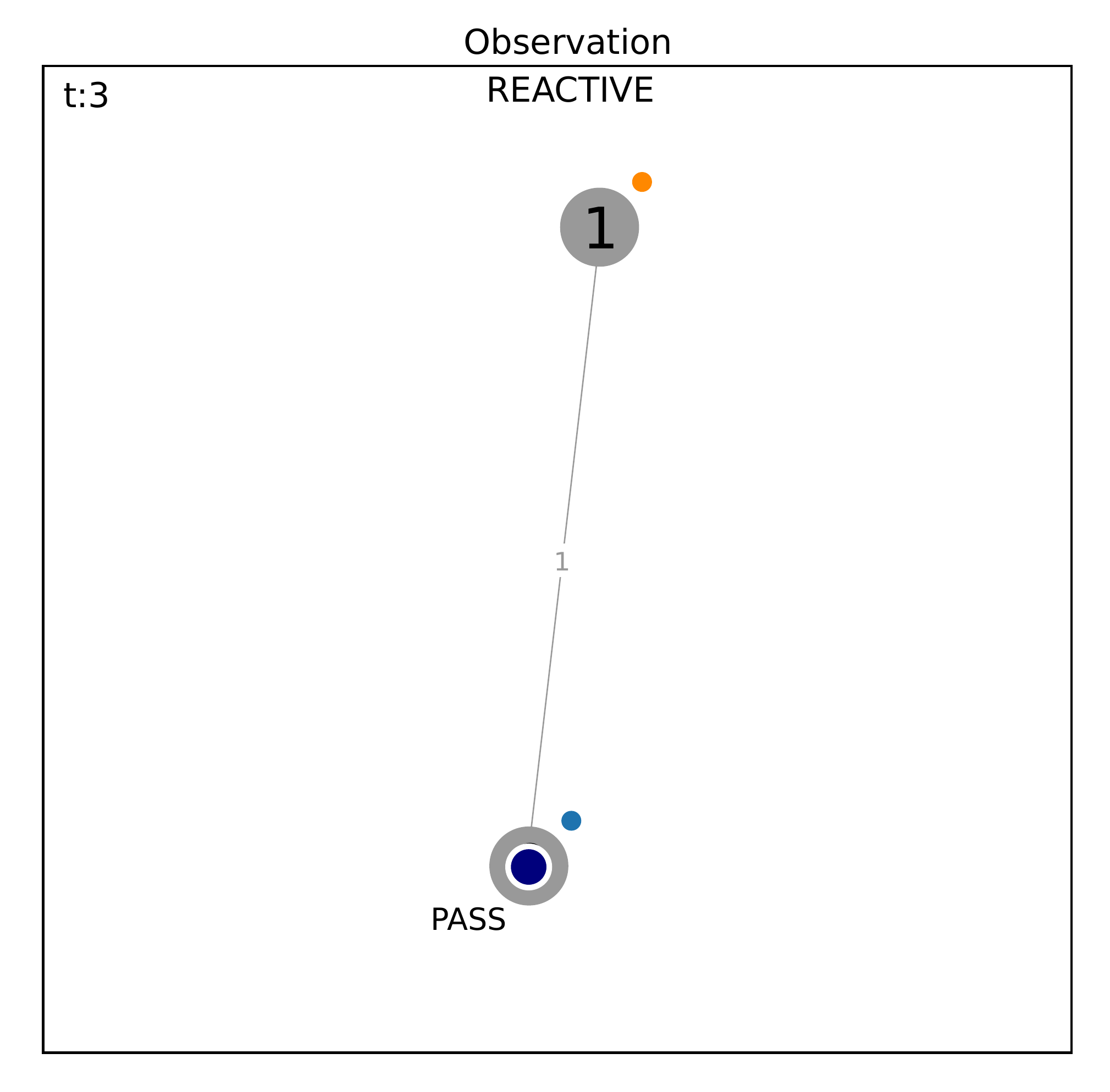}
 \caption{H-Reactive. $J(\pi)$:283.619}
 \label{subfig:pol-reac}
 \end{subfigure}
 \caption{Time step $t=3$ of one episode of M2-Q2Q3C2 (best seen in color) and average performance under each policy. \textit{Gray} nodes in the graph represent the machine locations numbered 0 (Q2) and 1 (Q3). The decision-maker is a \textit{navy blue} dot, and \textsc{pass}, \textsc{move}, \textsc{pm} and \textsc{cm}, correspond to wait, move to another location and preventive/corrective maintenance actions. Alerts are small \textit{orange} dots on top of machine nodes, otherwise \textit{blue} when no alert is present. }
 \label{fig:policy-diff}
\end{figure*}

In Figure \ref{subfig:pol-da}, we observe that the TMH makes a different decision than the greedy and reactive heuristic. Like the greedy heuristic, the TMH policy travels to location $1$ when the alert arrives. However, it decides to wait (\textsc{pass}) and postpone maintenance as it favors corrective maintenance under cost structure C2. Identical to the reactive heuristic, the TMH performs a CM action at $t=6$ after the machine has failed. In Figure \ref{subfig:pol-dqn}, we observe that given the same hidden network state, the DRL policy prescribes to remain at location 0 until $t=4$, implying that the agent has learned that the machine at location 1 needs \emph{at least} $3$ time steps to transition to the failed network state. Thus, the DRL agent chooses a less proactive behavior than the TMH while also postponing maintenance. In fact, at $t=4$, the DRL agent moves to location 0 and at $t=6$, the agent performs a CM action.

Until $t=6$, both the TMH and the DRL policies incur the same cumulative costs and at $t=12$, there are two alerts present in the network. Since both policies resulted in repairing the machine at location 0 at the same time, they observe identical alerts. At $t=12$, both engineers are located at machine 0; however, the DRL agent travels to location 1, while the TMH waits until a machine fails. The DRL agent proceeds to perform PM at location 1 at $t=13$, while at $t=14$ the TMH observes a failure at location 1, moves to that location and performs CM at $t=15$. 

On average, after receiving an alert, the machine at location 1 (Q3) fails more frequently than the one at location 0 (Q2). At $t=12$, the DRL agent prefers to travel to location 1 to avoid a failure while leaving the alert at location 0 unresolved. On the other hand, the TMH fails to balance the importance of the failure time distributions in this instance. Note that, at $t=15$ and afterwards, it prefers to postpone travel and preventive maintenance actions while waiting at location 0, resulting in higher costs as it rushes to repair failed machines at both locations.
For the remaining experiments, we observe that the DRL policies can manage the trade-off of costs and prioritize machines more efficiently than the other heuristics. The complete list of experiments comparing the proposed heuristics acting on the same environment is available online\footnote{\url{http://retrospectiverotations.com/dtmpa/dtmpa.html}}. 

\subsection{Managerial insights}
For larger networks, characterizing the behavior of the learned policies via a set of heuristic rules becomes challenging as such heuristics depend on more complex rules covering more machines, alerts and elapsed times. In general, we observe that the DRL agent moves proactively to healthy assets which carry high economic risk and learns to delay preventive maintenance adequately, sometimes incurring several extra downtime penalties. This implies that proactive handling on risky assets in a complex and uncertain maintenance environment pays off. Furthermore, it is not always optimal to react on every alert in the network. It is important to understand the behaviour of each machine after an issued alert and to estimate the near-future consequences.

For large asset networks ($|\mathcal{M}| = 6$), we observe that in one instance the DRL agent leaves assets with frequent failures (in terms of the distribution of $T^\textbf{\textrm{f}}_m$) unattended. This can be explained by the single engineer reaching its capacity limit. Thus, depending on the network size and degradation processes, a single engineer may not be sufficient to maintain the entire network without sacrificing reliability. 

Managers can employ such insights when determining the size of repair crews, considering the number of assets and their degradation processes. There is a cap on the number of assets that one engineer can maintain effectively. 

Each of the proposed methods in this work reflects on different assumptions regarding the observed information from the network and relates to different application scenarios: In case no historical information about the degradation of assets or the operational lifetimes is available, greedy or reactive policies can be easily implemented. If methods for estimating the remaining useful lifetime of assets exist, heuristic policies like the TMH can be employed to take into account the trade-off of costs of resolving the alerts in the network. When online observation is possible, or a simulation environment exists, policies can be learned via the DRL approach proposed in this work. In this case, an online system can take observations as inputs and generate instructions for a repair crew. As a result, the learned policies can inspire further heuristic developments.
\section{Conclusion and discussion}\label{sec: conclusion}
In this work, we introduce the dynamic traveling maintainer problem with alerts (DTMPA) for a network of modern industrial assets with stochastic failure times. We consider a realistic scenario where assets are part of an asset network and where alerts, triggered by real-time degradation monitoring devices, are utilized to make maintenance decisions. We propose a generic framework to this problem, where independent degradation processes model the alert as an early and imperfect indicator of failure. Therefore, the decision-maker only has access to partial degradation information to decide on cost-effective maintenance and travel actions.

To solve the problem, we propose three approaches, each requiring different information levels from the alerts. When alerts come with partial distributional information about the residual lifetime, we propose a class of greedy and reactive heuristics that rank alerts based on travel times, failure times or economical value. When full distributional lifetime information is available, we propose approximating the environment with a problem instance similar to the traveling maintainer problem. This method considers the alerts to determine a schedule that minimizes the near future maintenance costs. When alerts carry only information about the time since their generation and a simulation environment is present, we can learn policies via deep reinforcement learning aimed at approximating the state-action value distributions of long-term costs.

The results show that we can effectively replicate the performance of optimal policies when no information gap exists between alerts and the real underlying degradation. When an information gap is present and the number of assets in the network is small (up to four assets), the proposed methods yield effective policies that are close to the optimal policies under full information. When the optimal policy becomes computationally intractable, the learned policies that result in the lowest total expected discounted cost balance the failure risk and the maintenance costs better than the competing methods.

Future research directions include expanding the proposed framework to consider multi-component assets and asset networks serviced by multiple maintainers. The numerical study considers unit distances and repair times, which is a limitation that would be important to address in further work. The assumption of geometrically distributed degradation transition times in the experiments underlies the ability to solve to optimality some of the instances; this assumption can be relaxed within the more generic DTMPA framework, which we leave for future work.

\small
\paragraph{Acknowledgements} We would like to thank Verus Pronk and Jan Korst for the discussions that formed the basis for this paper. This work was partially carried out on the Dutch national e-infrastructure with the support of SURF Cooperative. This work was supported by the Netherlands Organisation for Scientific Research (NWO). Project: NWO Big data - Real Time ICT for Logistics. Number: 628.009.012. The work of Stella Kapodistria is supported by NWO through the Gravitation-grant NETWORKS-024.002.003. 
\bibliographystyle{elsarticle-harv}
\bibliography{references.bib}
\pagebreak
\section*{Supplementary material}
\beginsupplement
\appendix
\section{Theorems and definitions}\label{app:theorems}

\begin{definition}
For $m\in\mathcal{M}$, let $\{\tilde{X}_m(t),\ t \geq 0\}$ denote the \textit{(observed degradation) phase process} and let $\{N_m(t),\ t \geq 0\}$ denote a \textit{counting process} counting the cycles of machine installments (renewals defined by an as-good-as-new asset), namely $ N_m(t) = \sum_{s=0}^{t}\mathds{1}_{\{ \tilde{X}_m(s) = \tilde{x}^\textbf{\textrm{h}}_m \land \tilde{t}_m(s) = 0\}}.$ Here, $\tilde{t}_m(s)$ is the elapsed time since the last phase transition of machine $m$.
\end{definition}

\begin{definition}
For $m\in\mathcal{M}$, let $S^n_m = \underset{t\geq0}{\textrm{inf}}\{t: N_m(t) = n\}$ be the time until the start of the $n$-th cycle of machine $m$, where $n \geq 1$.
\end{definition}

The elapsed time since the last phase transition of machine $m$ can now be defined as 
$$\tilde{t}_m(t) = \underset{s\in\{0,\ldots,t-S^{N_m(t)}_m\} }{\textrm{sup}} \{s:\tilde{X}_m(t-s) = \tilde{X}_m(t)\}.$$
\begin{theorem}[Discounted cost of a policy $\pi^\textbf{L}$]\label{th: tmh}
Let $\gamma \in [0,1]$ be the discounting factor and $\tau \in \mathbb{N}$ be a maintenance deadline. The total expected discounted cost of the induced time-based maintenance policy $\pi^\textbf{L}(\tau)$, denoted by $J(\pi^\textbf{L}(\tau))$, satisfies
\begin{equation*}\label{}
J(\pi^{\textbf{L}}(\tau)) = \frac{ \bar{c}_m^{\textrm{CM}}\mathds{E}[\gamma^{T^\textbf{\textrm{a}}_m+T^\textbf{\textrm{f}}_m}\mathds{1}_{\{ T^\textbf{\textrm{f}}_m \leq \tau \}} \mid X_m(0) = x^\textbf{\textrm{h}}_m] + \bar{c}_m^{\textrm{PM}}\mathds{E}[\gamma^{T^\textbf{\textrm{a}}_m + \tau}\mathds{1}_{\{ T^\textbf{\textrm{f}}_m > \tau \}} \mid X_m(0) = x^\textrm{h}_m] }{ 1 -\gamma^{t^\textrm{CM}_m} \mathds{E}[\gamma^{T^\textbf{\textrm{a}}_m+T^\textbf{\textrm{f}}_m}\mathds{1}_{\{ T^\textbf{\textrm{f}}_m \leq \tau \}} \mid X_m(0) = x^\textrm{h}_m] - \gamma^{t^\textrm{PM}_m}\mathds{E}[\gamma^{T^\textbf{\textrm{a}}_m + \tau}\mathds{1}_{\{ T^\textbf{\textrm{f}}_m > \tau \}} \mid X_m(0) = x^\textrm{h}_m] },
\end{equation*}
where $\Bar{c}_m^{\textrm{CM}} = c_m^{\textrm{CM}}+ c_m^{\textrm{DT}}\frac{\gamma^{t_m^{\textrm{CM}}} - 1 }{\gamma-1} $ and $\Bar{c}_m^{\textrm{PM}} = c_m^{\textrm{PM}}+c_m^{\textrm{DT}}\frac{\gamma^{t_m^{\textrm{PM}}} - 1 }{\gamma-1}$.
\end{theorem}
\begin{proof}
The maintenance deadline induces a renewal reward process (cf. e.g, \citep[Ch. 7]{ross2014introduction}) where each new cycle starts after a maintenance period. Conditioning on the event $\{T^\textbf{\textrm{f}}_m \leq \tau\}$ or $\{T^\textbf{\textrm{f}}_m > \tau\}$, we find 
\begin{align*}
 J(\pi^{\textbf{L}}(\tau)) &=
 \mathds{E}[\gamma^{T^\textbf{\textrm{a}}_m+\textrm{min}\{T^\textbf{\textrm{f}}_m,\tau\}} ((\bar{c}_m^{\textrm{CM}} + \gamma^{t^\textrm{CM}_m}J(\pi^{\textbf{L}} )) \mathds{1}_{\{ T^\textbf{\textrm{f}}_m \leq \tau \}} + (\bar{c}_m^{\textrm{PM}} + \gamma^{t^\textrm{PM}_m}J(\pi^{\textbf{L}} ) )\mathds{1}_{\{ T^\textbf{\textrm{f}}_m > \tau \}}) \mid X_m(0) = x^\textrm{h}_m] \\
 &=\mathds{E}[\gamma^{T^\textbf{\textrm{a}}_m+T^\textbf{\textrm{f}}_m}(\bar{c}_m^{\textrm{CM}} + \gamma^{t^\textrm{CM}_m}J(\pi^{\textbf{L}} )) \mathds{1}_{\{ T^\textbf{\textrm{f}}_m \leq \tau \}} + \gamma^{T^\textbf{\textrm{a}}_m + \tau}(\bar{c}_m^{\textrm{PM}} + \gamma^{t^\textrm{PM}_m}J(\pi^{\textbf{L}} ) )\mathds{1}_{\{ T^\textbf{\textrm{f}}_m > \tau \}} \mid X_m(0) = x^\textrm{h}_m].
\end{align*}
Using the linearity of expectation and sorting the terms yields
\begin{align*}
 J(\pi^{\textbf{L}}(\tau))\left(1 -\gamma^{t^\textrm{CM}_m} \mathds{E}[\gamma^{T^\textbf{\textrm{a}}_m+T^\textbf{\textrm{f}}_m}\mathds{1}_{\{ T^\textbf{\textrm{f}}_m \leq \tau \}} \mid X_m(0) = x^\textrm{h}_m] - \gamma^{t^\textrm{PM}_m}\mathds{E}[\gamma^{T^\textbf{\textrm{a}}_m + \tau}\mathds{1}_{\{ T^\textbf{\textrm{f}}_m > \tau \}} \mid X_m(0) = x^\textrm{h}_m] \right) \\ = \bar{c}_m^{\textrm{CM}}\mathds{E}[\gamma^{T^\textbf{\textrm{a}}_m+T^\textbf{\textrm{f}}_m}\mathds{1}_{\{ T^\textbf{\textrm{f}}_m \leq \tau \}} \mid X_m(0) = x^\textrm{h}_m] + \bar{c}_m^{\textrm{PM}}\mathds{E}[\gamma^{T^\textbf{\textrm{a}}_m + \tau}\mathds{1}_{\{ T^\textbf{\textrm{f}}_m > \tau \}} \mid X_m(0) = x^\textrm{h}_m].
\end{align*}
Isolating $J(\pi^{\textbf{L}}(\tau))$ on the l.h.s. proves the result.\end{proof}

\newpage

\section{Traveling maintainer heuristic algorithms}\label{app:tmhalgos}

\begin{algorithm}[h]
\caption{Traveling maintainer heuristic}
\scriptsize
\begin{algorithmic}[1]
\Procedure{Construct initial schedule}{}\\
\textbf{Input: $\sigma \in \Sigma_{\mathcal{M}_t}$}
\State{Set $t^\sigma(1) = t$} \Comment{Current time}
\State{Set $t^\sigma(2) = t^\sigma(1)+\theta_{\ell(t),\sigma(1)}$} \Comment{Travel to first machine}
\State{Set $j = 3$}
\For{$i=1,\ldots,|\mathcal{M}_t|-1$}\Comment{Construct greedy schedule}
\State Set $t^\sigma(j) = t^\sigma(j-1) + \mathds{1}\left\{ t^\sigma(j-1) < \tilde{T}^\textbf{\textrm{f}}_{\sigma(i), N_{\sigma(i)}(t)} \right\}t^\textrm{PM}_{\sigma(i)} + \mathds{1}\left\{ t^\sigma(j-1) \geq \tilde{T}^\textbf{\textrm{f}}_{\sigma(i), N_{\sigma(i)}(t)} \right\}t^\textrm{CM}_{\sigma(i)} $ \Comment{Schedule PM or CM}
\State j = j+1
\State Set $t^\sigma(j) = t^\sigma(j-1) + \theta_{\sigma(i),\sigma(i+1)}$ 
\State j = j+1
\EndFor
\State Set $t^\sigma(j) = t^\sigma(j-1) + \mathds{1}\left\{ t^\sigma(j-1)
< \tilde{T}^\textbf{\textrm{f}}_{\sigma(|\mathcal{M}_t|), N_{\sigma(|\mathcal{M}_t|)}(t)} \right\}t^\textrm{PM}_{\sigma(|\mathcal{M}_t|)} + \mathds{1}\left\{ t^\sigma(j-1) \geq \tilde{T}^\textbf{\textrm{f}}_{\sigma(|\mathcal{M}_t|), N_{\sigma(|\mathcal{M}_t|)}(t)} \right\}t^\textrm{CM}_{\sigma(|\mathcal{M}_t|)} $\Comment{End time of last maintenance action}\\
\textbf{Output: $ \textbf{t}^\sigma = \left\{ t^\sigma(j) \mid j = 1,\ldots, 2|\mathcal{M}_t|+1 \right\}$}
\EndProcedure
\end{algorithmic}
\label{alg: tmh-1}
\end{algorithm}

\begin{algorithm}[htp]
\scriptsize
\caption{Traveling maintainer heuristic}
\begin{algorithmic}[1]
\Procedure{Optimize a schedule}{}\\
\textbf{Input: $(\sigma, \textbf{t}^\sigma)$ }
\State Set $\textbf{t}^\sigma_\textrm{opt} = \textbf{t}^\sigma$, $i = |\mathcal{M}_t|$
\State Set $\delta = \textbf{t}^\sigma_\textrm{opt}(2|\mathcal{M}_t|) - \tilde{T}^\textbf{\textrm{f}}_{\sigma(|\mathcal{M}_t|)} $
\Comment{Possible delay last repair}
\While{$\delta > 0 \land i > 0$}
\State Set $\textbf{t}^\sigma_\textrm{opt}(2i) = \textbf{t}^\sigma_\textrm{opt}(2i) + \delta$ \Comment{Delay repair}
\State Set $\textbf{t}^\sigma_\textrm{opt}(2i+1) = \textbf{t}^\sigma_\textrm{opt}(2i+1) + \delta$ \Comment{Delay next travel}
\State Set $i=i-1$
\State Set $\delta = \textrm{min}\left\{ \delta, \textbf{t}^\sigma_\textrm{opt}(2i)) - \tilde{T}^\textbf{\textrm{f}}_{\sigma(i)} \right\}$ \Comment{Possible delay previous repair}
\EndWhile
\State Compute $c(\sigma, \textbf{t}^\sigma_\textrm{opt}) = \sum_{s=t}^{\textbf{t}^\sigma_\textrm{opt}(2|\mathcal{M}_t|+1)} \gamma^s C_{\pi_s^\textbf{L}(\sigma, \textbf{t}^\sigma_\textrm{opt})}(h_s) $\\
\textbf{Output: } $(\sigma, \textbf{t}^\sigma_\textrm{opt}, c(\sigma, \textbf{t}^\sigma_\textrm{opt}))$ 
\EndProcedure
\end{algorithmic}
\label{alg: tmh-2}
\end{algorithm}

\newpage

\section{$n$-step quantile regression double Q-Learning algorithm}
\label{app:rlalgo}

 \begin{algorithm}[h]
\scriptsize
\begin{algorithmic}[1]
\State Initialize replay memory $\mathcal{D}$, $n$, $E$, $N_B$, $T$, $N$, $\alpha$, $P$, $\lambda$, $\epsilon$.
\State Initialize functions $\psi_w$, $\psi_{\bar{w}}$ with random weights $w = \bar{w}$ 
\For{episode $=1,\ldots, E$} 
\State Sample initial observable state $o_0$
\For {$t=0,\ldots, T-1$}

	\State Select a random action $u_t$ with probability $\epsilon$
	\State otherwise select $u_t = \arg \min_{u}\frac{1}{N}\sum^N_{i=1} \psi^i_{w}(o_t, u).$
	\State Execute $u_t$, observe $c_t$ and $o_{t+1}$

	\State Store transition $\left(o_t,u_t,c_t,o_{t+1}\right)$ in $\mathcal{D}$
	\If {$t \geq n$}
	\State Sample
	$\Big\{\left(o_j,u_j,C^{j:j+n},o_{j+n}\right)\Big\}^{N_B}_{j=1}$ from $\mathcal{D}$

	\State
	$y^{i^\prime}_{j} =
 \left\{
 \begin{array}{l l}
 C^{j:j+n^\prime} & \text{if } n^\prime \leq n \text{ and } o_{j+n^\prime} \text{ terminal} \\
 C^{j:j+n} + \gamma^n \psi^{i^\prime}_{\bar{w}}(o_{j+n}, u^*) & \text{otherwise}
 \end{array} \right.$
 \State where $u^* = \arg \min_{u^\prime} \sum^N_{i=1} \psi^i_{w}(o_{j+n}, u^\prime)$. 
	\State Take a gradient step acc. to Eq.(\ref{eq:qrddqn-nstep-loss}) w.r.t. $w$

	\EndIf
\EndFor
	\State Every $P$ episodes, update $\bar{w} = w$
\EndFor
\end{algorithmic}
\caption{ $n$-step quantile regression double Q-Learning ($n$QR-DDQN)}
\label{alg:dqn}
\end{algorithm}

\newpage

\section{Nomenclature} \label{app:nomenc}
\begin{table*}[h!]
 \begin{framed}
 \printnomenclature
 \nomenclature[01]{$t$}{Time step}
 \nomenclature[02]{$\mathcal{M}$}{Set of machines}
 \nomenclature[03]{$M$}{Number of machines}
 \nomenclature[04]{$m$}{Machine index}
 \nomenclature[05]{$c^\textrm{PM}_m$}{PM cost of machine $m$}
 \nomenclature[06]{$c^\textrm {CM}_m$}{CM cost of machine $m$} 
 \nomenclature[07]{$c^\textrm{DT}_m$}{Downtime cost of machine $m$} 
 \nomenclature[08]{$\theta_{ij}$}{Travel time between machines $i$ and $j$}
 \nomenclature[09]{$x^\textbf{h}_m$}{As-good-as-new state}
 \nomenclature[10]{$x^\textbf{\textbf{a}}_m$}{Alert state}
 \nomenclature[11]{$x^\textbf{\textbf{f}}_m$}{Failed state}
 \nomenclature[12]{$\mathcal{N}_m$}{State space of machine $m$}
 \nomenclature[13]{$x_m(t)$}{Degradation process of machine $m$} 
 \nomenclature[14]{$\tilde{x}^\textbf{h}_m$}{Healthy phase}
 \nomenclature[15]{$\tilde{x}^\textbf{\textbf{a}}_m$}{Alert phase}
 \nomenclature[16]{$\tilde{x}^\textbf{\textbf{f}}_m$}{Failed phase}
 \nomenclature[17]{$\mathcal{X}_m$}{Observation space of machine $m$}
 \nomenclature[18]{$T^{x_m}_m$}{Random transition time from state $x_m$ to $x_m+1$}
 \nomenclature[19]{$T^\textbf{\textrm{a}}_m$}{Time until the alert}
\nomenclature[20]{$T^\textbf{\textrm{f}}_m$}{Residual lifetime}
\nomenclature[21]{$h$}{Network state}
\nomenclature[22]{$o$}{Observed network state} 
\nomenclature[23]{$\Omega$}{Observation space} 
\nomenclature[24]{$\textbf{L}_i$}{Information level $i$}
\nomenclature[25]{$\mathcal{U}(h)$}{State-dependent action set}
\nomenclature[26]{$\pi^{\textbf{L}_i}$}{Policy under information level $i$}
 \end{framed}
\end{table*}

\newpage

\section{Deep reinforcement learning hyperparameters}
\label{app:drlparam}
During training of $n$QR-DDQN we employ a number of hyperparameters listed in Table \ref{tab:hyper}. 

\begin{table}[htp]
\centering
\caption{$n$QR-DDQN Hyperparameters}
\resizebox{0.80\textwidth}{!}{%
\begin{tabular}{c|l|c}
\toprule
Hyperparameter & Description & Value \\ \midrule
$T$ & length of an episode & 500 \\
$|\mathcal{D}|$ & size of the replay memory & 100000 \\
$E$ & number of episodes & 2000 \\
$P$ & number of episodes to update target NN & 30 \\
$N_B$ & mini-batch size & 32 \\
$\gamma$ & discount rate & 0.99 \\
$e_r$ & exploration fraction & 0.90 \\
$\epsilon_i$ & initial exploration probability (linear decay over $Ee_r$ episodes) & 0.1 \\
$\epsilon_f$ & final exploration probability & 0.005 \\
$\lambda$ & learning rate & $5 \times 10^{-4}$ \\
$n$ & steps to bootstrap in the $n$-step loss & 5 \\
$\alpha$ & $n$-step loss weight & 1 \\
$\kappa$ & Huber loss weight & 1 \\
$L$ & number of NN layers & 4 \\
$N$ & number of quantiles in QR & 51 \\
$M$ & number of machines in a network & $\{1, 2, 4, 6 \}$ \\
$d^l$ & hidden dimension of layer $l=1, \ldots, L-1$ & $64$ \\
$\sigma^1(x)$ & linear embedding layer $l=1$ & $x$ \\
$\sigma^l(x)$ & activation of hidden layer $l=2, \ldots, L-1$ & $\max(0, x)$ \\
 \bottomrule
\end{tabular}%
}
\label{tab:hyper}
\end{table}

\end{document}